\documentclass[11pt,letterpaper]{amsart}
\usepackage[utf8]{inputenc}
\usepackage[T1]{fontenc}
\usepackage{amsmath}
\usepackage{amsfonts}
\usepackage{amssymb}
\usepackage{graphicx}
\usepackage{amsthm}
\usepackage{tikz}
\usepackage{tikz-cd}
\usepackage{pgfplots}
\usepackage{mathtools}

\newcommand{\pare}[1]{\left(#1\right)}
\newcommand{\norm}[1]{\left\lVert#1\right\rVert}

\newtheorem{thm}{Theorem}[section]
\newtheorem{lem}[thm]{Lemma}
\newtheorem{cor}[thm]{Corollary}

\newtheorem{prop}[thm]{Proposition}
\newtheorem{dfn}[thm]{Definition}

\newtheorem{rem}[thm]{Remark}

\newtheorem{thmintro}{Theorem}

\newcommand{\TT}{\mathcal{T}}
\newcommand{\WW}{\mathcal{W}}

\newcommand{\Z}{\mathbf{Z}}
\newcommand{\C}{\mathbf{C}}
\newcommand{\R}{\mathbf{R}}

\newcommand{\N}{\mathbf{N}}

\newcommand{\QQ}{\mathcal{Q}}
 \newcommand{\Op}{\mathrm{Op}}
\newcommand{\del}{\partial}

\newcommand{\id}{\mathrm{id}}

\newcommand{\sgn}{\mathrm{sgn}}

\newcommand{\Sp}{\mathrm{Sp}}
\newcommand{\cS}{\mathcal{S}}

\newcommand{\Hom}{\mathrm{Hom}}
\newcommand{\im}{\operatorname{im}}

\DeclareMathOperator*{\colim}{colim}
\DeclareMathOperator{\diag}{diag}
\DeclareMathOperator{\sing}{Sing}
\DeclareMathOperator{\st}{st}

\DeclareMathOperator{\coind}{coind}
\DeclareMathOperator{\ind}{ind}

\newcommand{\supp}{\operatorname{supp}}
\newcommand{\Int}{\operatorname{Int}}
\numberwithin{equation}{section}

\usepackage{color}

\pgfplotsset{compat=1.17}

\def\co{\colon\thinspace}
\title[Twisted generating functions and nearby Lagrangians]{Twisted generating functions and the nearby Lagrangian conjecture}

\author[M. Abouzaid \and S. Courte \and S. Guillermou \and T. Kragh]{Mohammed Abouzaid \and Sylvain Courte \and Stéphane Guillermou \and Thomas Kragh}

\begin{document}

\begin{abstract}
We prove that, for closed exact embedded Lagrangian submanifolds of cotangent bundles, the homomorphism of homotopy groups induced by the stable Lagrangian Gauss map vanishes. In particular, we prove that this map is null-homotopic for all spheres. The key tool that we introduce in order to prove this is the notion of twisted generating function and we show that every closed exact Lagrangian can be described using such an object, by extending a doubling argument developed in the setting of sheaf theory. Floer theory and sheaf theory constrain the type of twisted generating functions that can appear to a class which is closely related to Waldhausen's tube space, and our main result follows by a theorem of Bökstedt which computes the rational homotopy type of the tube space.
\end{abstract}

\maketitle

\tableofcontents

\section{Introduction}

\subsection{The Gauss map for nearby Lagrangians}
\label{sec:gauss-map-nearby}

The nearby Lagrangian conjecture predicts that any closed exact Lagrangian
submanifold $L$ in the cotangent bundle $T^* M$ of a closed manifold $M$ is Hamiltonian isotopic to the
zero-section; we refer to such a submanifold $L$ as a \emph{nearby Lagrangian}.
The conjecture is wide open in general and presently known to hold only for $M=S^1$ (where it is elementary), $M=S^2$ (\cite{Hind2012}) and $M=T^2$ (\cite{dimitroglou_rizell_lagrangian_2016}).
A first obstruction is that the projection $\pi\colon L \to M$ may not even be homotopic to a diffeomorphism,
or even worse that $L$ may not be diffeomorphic to $M$.
The first and fourth authors have shown in \cite{Kragh2013} and \cite{Abouzaid2012a} that $\pi$ is at least a homotopy equivalence, a result reproved by
the third author in \cite{guillermou_quantization_2012}. In \cite{AbouzaidKragh2018}, the first and fourth authors further proved that $\pi$ is a simple homotopy equivalence. Constraints on the smooth structure of $L$ have also been
discovered, e.g. by the first author in \cite{Abouzaid2012} in the case where $M$ is a sphere (see also \cite{ekholm_lagrangian_2015}).

If the conjecture holds, then the tangent bundle of $L$ and the cotangent fibres define homotopic sections of the restriction of the Lagrangian Grassmannian bundle of $T^*M$ to $L$. This implies the triviality of the \emph{stable Gauss map} $L\to \Lambda_0(\C^\infty)$ (see Definition~\ref{dfn:stablegaussmap},
the stable Lagrangian Grassmannian $U/O$ is most often denoted $\Lambda_0(\C^\infty)$ in this paper). Our main result in this regard is the following:
\begin{thmintro}\label{thmintro:gaussmap}
Let $M$ be a closed manifold and $L$ a closed exact Lagrangian submanifold of $T^* M$, then the stable Gauss
map $L\to \Lambda_0(\C^\infty)$ vanishes on all homotopy groups.
\end{thmintro}
Note that the fourth author in \cite{kragh_generating_2018} proves the triviality of the stable Gauss map
in the case where $M=S^n$ under the additional assumption that $L\subset T^* M$ coincides with the zero-section
above a neighborhood of some point in $M$, which is a particular case of Theorem~\ref{thmintro:gaussmap}.

In order to relate this to previous results, recall that the obstruction to the triviality of the Gauss map on the $2$-skeleton of
$L$ is given by two cohomology classes: the first Maslov class $\mu_1 \in H^1(L;\Z)$ and a relative Stiefel-Whitney class
$w_2\in H^2(L;\Z/2)$. These were shown to vanish by the first and fourth authors using Floer theory,
and later by the third author using microlocal sheaf theory.

Regarding higher obstruction, the first and fourth authors already proved a vanishing result
(see~\cite[Prop. 2.2]{AbouzaidKragh2013}), namely, after looping, the composition of the stable Gauss map
$L\to U/O =B(\Z\times BO)$ and the map $B(\Z\times BO) \to B(\Z\times BG)$ induced by the $J$-homomorphism
is nullhomotopic. Jin proved later in~\cite{jin_2020} that the same vanishing holds without looping. We will
give an alternate proof of this improvement in Section~\ref{sec:tubespaces}.
 
\begin{thmintro}\label{thmintro:gaussJ}
Let $M$ be a closed manifold and $L$ a closed exact Lagrangian submanifold of $T^* M$, then the composition
$L\to U/O\to B(\Z\times BG)$ of the
stable Gauss map  with (a delooping of) the $J$-homomorphism $U/O=B(\Z\times BO) \to B(\Z\times BG)$ is nullhomotopic.
\end{thmintro}

Both of these results are proved using generating functions, a classical tool
in symplectic and contact topology, which will be the main subject of this paper.
Moreover, the only essential input from Floer theory or microlocal sheaf theory that we use is the result that $\pi$
is a homotopy equivalence.

\subsection{Twisted generating functions}
\label{sec:twist-gener-funct}

The relevance of generating functions to the study of Lagrangians can be first seen from the following theorem of Giroux and Latour (Theorem~\ref{thm:girouxlatour}):
for a closed manifold $L$, a Legendrian immersion $L\to J^1 M$ admits a \emph{generating function}
(see Definition~\ref{dfn:gf}) if and only if the stable Gauss map $L\to \Lambda_0(\C^\infty)$ is trivial.
Let us remark that we are not able to prove that any exact Lagrangian embedding has a generating function so we cannot use the "only if"
part to prove that the stable Gauss map is trivial. To circumvent this issue, we introduce a weaker object that we
call a \emph{twisted} generating function (see Definition~\ref{dfn:tgf}) and prove the following theorem.

\begin{thmintro}\label{thmintro:twistedgirouxlatour}
Let $M$ and $L$ be closed manifolds. A Legendrian immersion $L\to J^1 M$ admits a twisted generating function
if and only if the stable Gauss map $L\to \Lambda_0(\C^\infty)$ factors up to homotopy through $\pi\colon L \to M$.
\end{thmintro}

Note that this condition on the stable Gauss map holds for nearby Lagrangians since we know that $\pi$
is a homotopy equivalence. Our concrete model for a twisted generating function is a (directed) open cover $(M_i)$ of $M$
and a collection of generating functions $(f_i)$ defined on open sets of $M_i\times \R^{n_i}$
which glue on double intersections up to a $1$-cocycle $(q_{ij})$ of fiberwise non-degenerate quadratic forms,
namely
\[f_i \oplus q_{ij}=f_j\text{ and } q_{ij}\oplus q_{jk}=q_{ik}.\]
Note that the cocycle $q_{ij}$ is valued in the space $\QQ$ of non-degenerate quadratic forms (with eigenvalues $\pm 1$), which is only a topological monoid.
The theory of principal $\QQ$-bundles and associated bundles is not completely standard
and we develop some bits of the theory that we need in Appendix~\ref{app:monoids-and-bundles}, with inspiration from \cite{baasdundasrognes_2004}.
It turns out that the integers $(n_i)$ and cocycle $(q_{ij})$ are classified by a map $M\to \Lambda_0(\C^\infty)$,
so are perfectly suited to encode the stable Gauss map. We note also that the notion of twisted generating function
is inspired by that of twisted sheaf used by the third author in \cite{guillermou_quantization_2012}. Theorem~\ref{thmintro:twistedgirouxlatour} is proved in
Section~\ref{sec:girouxlatour}.

Another issue is that the generating functions produced in this way (twisted or not)
are inexploitable because they may not be well-behaved at infinity. For example, any attempt to assign a homology group to a general generating function fails as the moduli spaces of gradient flow lines may escape to infinity, obstructing the Morse operator from defining a differential. However, for an \emph{embedded} Lagrangian
we give a procedure to convert a (twisted) generating function into another one which is well-behaved at infinity,
leading to the following result, proved in Section~\ref{sec:doubling}.

\begin{thmintro}\label{thmintro:tgftube}
Let $M$ be a closed manifold and $L$ a closed exact Lagrangian submanifold of $T^* M$. Then there exists a
twisted generating function of tube type which tube generates $L$.
\end{thmintro}

Here a generating function of tube type is essentially (and up to stabilizations) a generating function which is linear at infinity such that the regular sub-level set $\{f\leq 0\}$ has the homotopy type of a sphere (see Definition~\ref{dfn:tgftube}). The notion of ``tube generation'' (see Definition~\ref{dfn:tgftube}) means that we only consider what $f$ restricted to this sub-level set generates.

The key point is that we get the desired control of the Morse theory of the generating function to extract topological and homological information about $L$ from it. Generating functions of tube type appear in our context essentially as a refinement of the fact that the homological intersection number of $L$ with each cotangent fibre is $1$ (this is proved using either Floer theory or microlocal sheaves, in the course of showing the stronger statement discussed above that the projection map from $L$ to $M$ is a homotopy equivalence). The tube condition is related  to the tube space $\TT_\infty$ arising in Waldhausen's
manifold approach to algebraic K-theory of spaces (see \cite{waldhausen_algebraic_1982}): a \emph{tube} is a codimension $0$ submanifold of $\R^{m+1}$
which corresponds to the attachment of a single trivial handle on top of $\R^m\times (-\infty,0]$.
Stably, the space of such tubes was proved by Bökstedt in \cite{bokstedt_rational_1984} to be rationally homotopy equivalent
to $BO$. This gives important constraints on how a tube bundle may be twisted,
leading to Theorem \ref{thmintro:gaussmap}. This is explained in Section~\ref{sec:tubespaces} where we
reformulate Bökstedt's theorem in terms of function spaces arising naturally from our construction
in Section~\ref{sec:doubling}. We note that Bökstedt's theorem is also crucial in \cite{kragh_generating_2018}.

Finally we observe that in the case where $M$ is a homotopy sphere, Theorem~\ref{thmintro:gaussmap} gives
the triviality of the stable Gauss map and our method provides a genuine generating function.

\begin{thmintro}\label{thmintro:homotopysphere}
  Let $M$ be a homotopy sphere and $L$ a closed exact Lagrangian submanifold of $T^* M$. Then the stable Gauss map of $L$ is null homotopic and there exists a generating function of tube type which tube generates $L$.
\end{thmintro}

\begin{rem}
  While we do not discuss applications to the diffeomorphism type of nearby Lagrangian submanifolds in this paper, forthcoming work of Abouzaid, \'Alvarez-Gavela, Courte, and Kragh, uses its main result, stated as Theorem \ref{thmintro:tgftube} below, as the starting point to prove that nearby exact Lagrangian submanifolds have associated normal invariant (measuring the difference between the tangent spaces of the Lagrangian and that of the base) which is always $2$-torsion (in fact, the main result is sharper, and involves the action of the Hopf map). This provides the first general result constraining the diffeomorphism type of Lagrangians, beyond their simple homotopy type, and relies essentially on the notion of \emph{twisted generating function} introduced below. Recently, Porcelli and Smith announced an alternative proof of a similar result using homotopical refinements of Floer homology. However, their approach still uses the results in this paper as a starting point.
\end{rem}

\paragraph{Acknowledgements}
The second and third authors would like to thank Emmanuel Giroux for teaching them his proof of Bott periodicity
and for his stimulating interest in this work. The first and fourth authors would like to thank Tobias Ekholm for insightful discussion on the subject.

We would also like to thank Trygve Poppe Oldervoll for pointing out an omission in a previous version of the proof of Lemma~\ref{lem:translation}. 

The first author was supported by the Simons Foundation through its ``Homological Mirror Symmetry'' Collaboration grant SIMONS 385571, and by NSF grants DMS-1609148, and DMS-1564172.

The second and third authors are partially supported by the ANR project MICROLOCAL
(ANR-15CE40-0007-01).

The fourth author was supported by the VR project ``2018-04237'' titled ``Lagrangian submanifolds, Algebraic K-theory and Quantum Physics''.

\section{A twisted generalization of a theorem of Giroux and Latour}\label{sec:girouxlatour}

\subsection{Twisted generating functions}
Let $M$ and $L$ be manifolds and $\varphi\times z\colon L\to J^1 M=T^* M\times \R$ a Legendrian immersion. We denote the projection from the cotangent bundle to the base by $\pi_M\colon T^*M \to M$ and its composition with the immersion of $L$ by $\pi=\pi_M\circ \varphi \colon L \to M$. In the symplectic vector bundle $E=\varphi^* (TT^*M)$ over $L$, there are two Lagrangian subbundles: 
\begin{itemize}
\item the \emph{vertical} subbundle $V=\ker d\pi_M$,
\item the tangent bundle of $L$, denoted $G$ for \emph{Gauss} section.
\end{itemize}
They can equivalently be considered as sections of the Lagrangian Grassmannian bundle $\Lambda_0(E)\to L$ associated to $E$ (the subscript $0$ in $\Lambda_0$ is just here for coherence with later notations).

\begin{dfn}\label{dfn:gf}
A \emph{generating function} over a manifold $M$ is a triple $(n,U,f)$ where $n\in \N$,
$U$ is an open set in $M\times \R^n$ and $f\colon U\to \R$
is a smooth function such that the fiberwise derivative $(x,v)\in U\mapsto \frac{\del f}{\del v}(x,v)\in (\R^n)^*$
vanishes transversely.

The submanifold $\Sigma_f=\{(x,v)\in U, \frac{\del f}{\del v}(x,v)=0\}$, called the \emph{singular set} of $f$,
then has a natural Legendrian immersion
$i_f\colon \Sigma_f \to J^1 M$ given by $(x,v)\mapsto (x,\frac{\del f}{\del x}(x,v),f(x,v))$.

We say that a Legendrian immersion $\varphi\times z : L \to J^1 M$ \emph{admits a generating function}
if there exists a generating function $(n,U,f)$ over $M$ and a diffeomorphism
$\psi \colon L\to \Sigma_f$ such that $i_f\circ \psi=\varphi\times z$. A generating function for $\varphi \times z$
is then a quadruple $(n,U,f,\psi)$. We also say that $(n,U,f,\psi)$ generates the Lagrangian immersion $\varphi$.
\end{dfn}

\begin{rem}
As discussed in the introduction, this definition of a generating function is rather weak since we do not impose conditions at infinity.
We cannot run any global Morse theoretic argument without more hypotheses on $U$ and $f$.
\end{rem}

\begin{rem}
Note that for Lagrangian immersions we have to consider generating $1$-forms instead of generating functions. We avoid this here by considering only Legendrian immersions (i.e. lifts of \emph{exact} Lagrangian immersions) but this is in no way crucial for the results of this section.
\end{rem}

The question of the existence of generating functions was settled independently by Giroux and Latour (see \cite{Giroux1990} and \cite{Latour1991}):

\begin{thm}[Giroux, Latour]\label{thm:girouxlatour}
Let $M$ be a manifold and $L$ a closed manifold. The Legendrian immersion $\varphi \times z\colon L\to J^1 M$ admits a generating function
if and only if the sections $G$ and $V$ are stably homotopic.
\end{thm}

Here stably homotopic means that for some large enough integer $k$, the sections $\R^k\times G$ and $\R^k\times V$ of the Lagrangian Grassmannian
bundle 
\[\Lambda_0(\C^k \times E)\to L\]
are homotopic. This condition can be rephrased in terms of the stable Gauss map of the Lagrangian immersion $\varphi$ which we now define. Roughly it is a map that measures the stable difference between the sections $G$ and $V$. For this definition we assume that $L$ is compact.
Pick a Lagrangian subbundle $H$ in $E$ which is transverse to $V$. The symplectic form then gives a canonical isomorphism $H\oplus H^*\to E$ which takes $H^*$ to $V$.
Since $L$ is compact we may find a vector bundle $H'\to L$ such that $H'\oplus H$ is a trivial vector bundle. Writing $m$ for the dimension of this bundle, we pick a trivialization $H'\oplus H \to L\times \R^m$. We then have the following canonical isomorphisms of symplectic vector bundles:
\begin{equation}
H' \oplus (H')^* \oplus E \simeq H'\oplus (H')^* \oplus H \oplus H^* \simeq (H'\oplus H) \oplus (H'\oplus H)^* \simeq L\times \C^m.
\end{equation}
The Lagrangian subbundle $H'\oplus G$ of $H'\oplus (H')^*\oplus E$ can now be viewed as a subbundle of the trivial symplectic vector bundle, hence simply as a map $L\to \Lambda_0(\C^m)$, with target the Grassmanian of linear Lagrangian subspaces of $\C^m$.
We may now take a colimit as $m$ goes to infinity and define
\begin{equation}
\Lambda_0(\C^\infty)=\colim_{m\to \infty}\Lambda_0(\C^m)
\end{equation}
with respect to the maps $\Lambda_0(\C^m)\to \Lambda_0(\C^{m+1})$ which takes a Lagrangian $G \subset \C^m$ to $\R\times G \subset \C \times \C^m=\C^{m+1}$.
\begin{dfn}\label{dfn:stablegaussmap}
The \emph{stable Gauss map} of the Lagrangian immersion $\varphi$ is the composition
\[g_\varphi\colon L\to \Lambda_0(\C^m)\to \Lambda_0(\C^\infty).\]
\end{dfn}
The homotopy class of this map is independent of the choices that we have made. First the space of Lagrangians transverse to $V$ in $E$ is contractible so the choice of $H$ is irrelevant. The choice of complement $H'$ and trivialization $H'\oplus H \to L\times \R^m$ is also the same as a monomorphism $H\to L\times \R^m$ and a choice of complement of its image. The space of such monomorphisms is contractible in the limit $m\to \infty$ and the space of complements of the image is also contractible. This allows us to rephrase the condition in Theorem~\ref{thm:girouxlatour}.

\begin{prop}\label{prop:trivialgaussmap}
The stable Gauss map of a Lagrangian immersion is nullhomotopic if and only if the sections $G$ and $V$ of $\Lambda_0(E)$
are stably homotopic.
\end{prop}
\begin{proof}
If $G$ and $V$ are stably homotopic, then with the previous notations $\R^k\times (H'\oplus G)$ is homotopic to $\R^k\times (H'\oplus V)$ for some $k$,
which is further homotopic to $\R^k \times (H'\oplus H)=L\times\R^{k+m}$. Hence $g_\varphi$ is nullhomotopic already in $\Lambda_0(\C^{k+m})$ and thus also in $\Lambda_0(\C^\infty)$.

Conversely if $g_\varphi$ is nullhomotopic, then for some large enough $m$ and some complement $H'$ in $\R^m$, the subbundle $H'\oplus G$
of $L\times \C^m $ is homotopic to $H'\oplus H=L\times \R^m$. Hence $H\oplus H'\oplus G$ is homotopic to $H\oplus H'\oplus H$. But $H\oplus H'\simeq H'\oplus H=L\times \R^m$, so we obtain that $\R^m\times G$ is homotopic to $\R^m \times H$, which is further homotopic to $\R^m \times V$.
\end{proof}

We now define the notion of twisted generating function aimed at dealing with some Legendrian submanifolds with non trivial stable Gauss map.
The twisting will be done by families of non-degenerate quadratic forms, which we introduce first:

\begin{dfn}\label{dfn:monoid_Q}
For $n\in \N$, we define $\QQ_n$ to be the manifold of all non-degenerate quadratic forms $q$ on $\R^n$ with eigenvalues $\pm 1$, namely
\[q\circ u(x_1,\dots,x_n)=\sum_{i=1}^n \pm x_i^2\] for some linear isometry $u$ of $\R^n$. We also define the disjoint union
\[\QQ=\bigsqcup_{n\in \N} \QQ_n.\]
For the purpose of stabilization, we distinguish the element $h\in \QQ_2$ defined by
\[h(x_1,x_2)=2x_1x_2\]
(the factor $2$ is needed for $h$ to have eigenvalues $\pm 1$).

The direct sum $\oplus$ equips $\QQ$ with the structure of a unital monoid, with unit the unique quadratic form on $\R^0$.

The map $\dim : \QQ\to \N$ which maps $q\in \QQ_n$ to $n$ is a monoid map and thus $\N$ and $\Z$ inherit actions of $\QQ$: for $q\in \QQ$ and $n\in \Z$,
we write $n\cdot q=\dim(q) + n$.
\end{dfn}

\begin{rem}
The restriction to quadratic forms with eigenvalues $\pm 1$ is not essential as the bigger space of all non-degenerate quadratic forms
deformations retracts onto $\QQ$. However this slightly simplifies some constructions in Section~\ref{sec:doubling} where explicit bounds and cut-off functions are used.
\end{rem}

\begin{dfn}\label{dfn:tgf}
A \emph{twisted generating function} over a manifold $M$ consists of the following data:
\begin{itemize}
\item a directed open cover $(M_i)_{i\in I}$ of $M$ (see Definition \ref{dfn:directedcover}),
\item for each $i$, a generating function $(n_i, U_i, f_i)$ over $M_i$,
\item for each $i<j$, a smooth map $q_{ij}:M_{ij}\to \QQ$, defined on the intersection $M_{ij} = M_i \cap M_j$,
\end{itemize}
such that
\begin{itemize}
\item for all $i<j$, $U_j\cap (M_{ij}\times \R^{n_j})\subset (U_i\cap (M_{ij}\times \R^{n_i})) \times \R^{n_j-n_i}$ and
$f_i\oplus q_{ij}=f_j$ on $U_j\cap (M_{ij}\times \R^{n_j})$ (in particular, $\dim(q_{ij})=n_j-n_i$ at least
over the image of the projection $U_j\cap (M_{ij}\times \R^{n_j})\to M_{ij}$),
\item for all $i<j<k$, $q_{ij}\oplus q_{jk}=q_{ik}$ on the triple intersection $M_{ijk} = M_i \cap M_j \cap M_k$.
\end{itemize}
We abbreviate this data as a quadruple $(M_i,n_i,f_i,q_{ij})$.
\end{dfn}

Note that if $M_{ij}$ is empty, the datum of $q_{ij}$ is void and we allow the possibility that $i<j$ and $n_i>n_j$ (but see Definition~\ref{dfn:well-order}).
For each $i\in I$, we have a Legendrian immersion $i_{f_i}:\Sigma_{f_i}\to J^1 M_i$. For $i<j$, since $d q_{ij}$ vanishes
only at the origin ($q_{ij}$ is non-degenerate), the relation $f_i\oplus q_{ij}=f_j$ imposes the following equality of
subsets of $M_{ij}\times \R^{n_j}$:
\[\Sigma_{f_j}\cap (M_{ij}\times \R^{n_j})=(\Sigma_{f_i} \times 0)\cap U_j\cap (M_{ij}\times \R^{n_j}),\]
and the coincidence of the Legendrian immersions on this subset:
\[i_{f_j}=i_{f_i}\circ \pi_{ij}\]
where $\pi_{ij}:M_{ij}\times \R^{n_j}\to M_{ij}\times \R^{n_i}$ is the projection.
This implies that the pieces $\Sigma_{f_i}$ glue into an abstract manifold $\Sigma_f$ with a Legendrian
immersion $i_f:\Sigma_f\to J^1 M$. If the integers $n_i$ are bounded above by some integer $n$, then $\Sigma_f$ is naturally embedded
in $M\times \R^n$ in such a way that 
\[\Sigma_f\cap (M_i\times \R^n)=\Sigma_{f_i}\times 0\subset (M_i\times \R^{n_i})\times \R^{n-n_i}.\]

Note that $i_f$ may not be proper (for instance $U_2,\dots,U_k$ may be empty and only $\Sigma_{f_1}$ is non-empty)
but in our situation the Legendrian immersion will be prescribed and surjective over $M$ so this will force all $U_i$'s to be non-empty.

The geometrically meaningful datum associated to a twisted generating function is its equivalence class under the procedure of refinement which we now explain: if $((M_{i'})_{i'\in I'}, \sigma: I'\to I)$ is a refinement of $(M_i)_{i\in I}$ in the sense of Definition \ref{dfn:refinment_directed_cover}, then we may pull back
a twisted generating function over $(M_i)_{i\in I}$ to one over $(M_{i'})_{i'\in I'}$ by setting
for $i'\in I'$, $n_{i'}=n_{\sigma(i')}$, $U_{i'}=U_{\sigma(i')}\cap (M_{i'}\times \R^{n_{\sigma(i')}})$,
$f_{i'}=f_i$ restricted to $U_{i'}$ and for $i'<j'$, $q_{i'j'}=q_{\sigma(i')\sigma(j')}$ (where $q_{ii}$ is interpreted as the unit in $\QQ_0$) restricted to $M_{i'j'}$.
The manifold $\Sigma_{f'}$ coincides with $\Sigma_f$ and has the same Legendrian immersion
to $J^1 M$.

Refinement allows us to often restrict our attention to the following class of twisted generating functions, which will be particularly convenient to simplify various inductive arguments:
\begin{dfn}\label{dfn:well-order}
  A twisted generating function $(M_i,n_i,f_i,q_{ij})_{i \in I}$ is called well-ordered if $I$ is finite, the order on $I$ is total and $n_i\leq n_j$ when $i<j$.
\end{dfn}

\begin{lem} \label{lem:well-order}
  Any twisted generating function on a compact manifold $M$ can be refined and then the order can be changed to yield an equivalent well-ordered twisted generating function.
\end{lem}

\begin{proof}
  Lemma~\ref{lem:total-order} shows that we can refine any twisted generating function to have the order on $I$ total, by compactness we can assume $I$ is finite. It is then clear  that $M_i \cap M_j=\varnothing$ if $n_i>n_j$ and $i<j$;  hence we can define a new total
order $\leq'$ by: $i<'j$ if $n_i<n_j$ or if $n_i=n_j$ and $i < j$. This new order is compatible in the sense that it makes $(M_i,q_{ij},f_i)_{i \in I}$ a twisted generating function. Both orders on $I$ have a common partial order which also makes it a twisted generating function - hence they are equivalent.
\end{proof}


It is classical in the theory of generating functions to stabilize by adding a fiberwise non-degenerate quadratic
form. The novelty here is that we use a $1$-cocycle of such quadratic forms to stabilize.
In the language of the Appendix, the directed open cover $(M_i)_{i\in I}$ together with the integers $n_i$ and
quadratic forms $q_{ij}$ define a $\QQ$-twisted map from $M$ to $\Z$. As we explain in Appendix~\ref{app:monoids-and-bundles},
the set of $\QQ$-twisted maps from $M$ to $\Z$ up to the appropriate equivalence relation
is in bijection with homotopy classes of maps from $M$ to a classifying space $|B(\Z,\QQ)|$ (see Proposition~\ref{prop:QF-bundles}).
In Subsection~\ref{subsec:girouxbott}, we will construct a homotopy equivalence 
\begin{equation}\label{eq:bottequivalence}
|B(\Z,\QQ)|\simeq \Lambda_0(\C^\infty).
\end{equation}

\begin{dfn}
We say that a Legendrian immersion $\varphi\times z \colon L \to J^1 M$ \emph{admits a twisted generating function
with twisting classified by} a map $h:M\to |B(\Z,\QQ)|$ if there exists a twisted generating function
$(M_i,n_i,f_i,q_{ij})$ and a diffeomorphism $\psi:L\to \Sigma_f$ such that $i_f\circ \psi=\varphi\times z$
and the $\QQ$-twisted map from $M$ to $\Z$ given by $(M_i,n_i,q_{ij})$ is classified by $h$.
\end{dfn}

We may now state the main result of this section which is a slightly more precise version of Theorem~\ref{thmintro:twistedgirouxlatour}.
\begin{thm} \label{thm:twistedgirouxlatour}
Let $M$ and $L$ be closed manifolds. A Legendrian immersion $\varphi\times z \colon L\to J^1 M$ admits a
twisted generating function with twisting classified by a map $h:M\to |B(\Z,\QQ)|$ if and only if
$g_\varphi$ is homotopic to $h\circ\pi$ (under the homotopy equivalence \eqref{eq:bottequivalence}).
\end{thm}

The proof of Theorem~\ref{thm:twistedgirouxlatour} will be given in Subsection~\ref{subsec:existencetgf}.

\begin{rem}
Compared to Theorem~\ref{thm:girouxlatour}, Theorem~\ref{thm:twistedgirouxlatour} has the additional
hypothesis that $M$ is closed, this is certainly not essential but makes the proof easier.
However compactness of $L$ is crucial in both results.
\end{rem}

One can recover Theorem~\ref{thm:girouxlatour} from the above theorem
applied with $h$ a constant map, together with the fact that a twisted generating function with trivial twisting (i.e. classified by a nullhomotopic map)
on a compact manifold  can be turned into a genuine generating function. We do not prove this claim
as we do not need it but the following result (which we need elsewhere anyway) is a step in this direction. 
For its statement, we use the notion of negligible shrinking as in Definition \ref{neg-shrink}, and in its proof, we use germs of smooth function as explained in Remark~\ref{rem:smoothextension} and Remark~\ref{warning}.

\begin{lem}\label{lem:trivialQbundle} 
  Let $(n_i,q_{ij})$ be a smooth well-ordered $\QQ$-twisted map from $M$ to $\N$ defined on a cover $(M_i)_{i\in \{1,\dots,k\}}$. Assume that  that $n_j-n_i$ is even for all $i<j$, and that there exist
smooth families of maps $\{q_{ij}^t\}_{i < j}$ parametrised by $t\in [0,1]$, such that the homotopies $q_{ij}^t$'s define a $\QQ$-bundle for
  all $t$ and satisfy the boundary conditions 
\[q_{ij}^0 = q_{ij}, \quad q_{ij}^1 = h^{\tfrac{n_j-n_i}{2}}\]
(recall $h(x,y) = 2xy$).
Then up to negligible shrinking there exist smooth maps $Q_i\colon M_i\to \QQ$, $1\leq i \leq k$,
  such that $q_{ij} \oplus Q_j = Q_i$ on $M_{ij}$ for all $i<j$. 
\end{lem}
\begin{proof}
Up to reparametrizing the homotopy $q_{ij}^t$ we can assume that $q_{ij}^t=h^{\frac{n_j-n_i}{2}}$ for $t$ near $1$.
We set
  $n'_i = (n_k -n_i)/2$. Up to negligible shrinking we may assume that the
  $q_{ij}^t$'s are defined on $\Op(\overline{M}_{ij})$ (we recall that the notation $\Op$ refers to an open neighbourhood).
We will define smooth maps $Q_i^t\colon\Op(\overline{M_i})\to \QQ$
  for all $t\in [0,1]$, such that $Q_i^t=h^{n'_i}$ for $t$ near $1$. We do this downwards
  inductively and define $Q_i^t$ starting with $Q_k^t=0$ and thus define $Q_i^t$
  assuming that all $Q_j^t$ for $j>i$ are already defined. Consider the closed
  subset
  \begin{align*}
    M_i' = \bigcup_{j>i} \overline{M}_{ij} \times [0,1] \cup \overline{M_i} \times \{1\} \subset \overline{M_i} \times [0,1] .
  \end{align*}
  On a neighborhood of $M'_i$ we already know what $Q_i^t$ is (for the equations
  $q_{ij}^t \oplus Q_j^t = Q_i^t$ for all $t$ and $Q_i^t = h^{n'_i}$ for $t$ near $1$ to hold).
By flowing an appropriate cut-off of the vector field $-\frac{\del}{\del t}$
which vanishes near $M'_i$, we can now construct a smooth map $\phi:M\times[0,1]\to M\times [0,1]$ which is the identity on $\Op(M'_i)$
and which takes $\overline{M_i}\times[0,1]$ into $\Op(M_i')$ (the two open neighbourhoods appearing here are nested) and we extend $Q_i$ to $\Op(\overline{M_i}\times[0,1])$
by setting $Q_i=Q_i\circ\phi$.
\end{proof}

\subsection{Symplectic reduction and Bott periodicity}\label{subsec:girouxbott}

Our main goal here is to explain the homotopy equivalence \eqref{eq:bottequivalence}.
This is essentially a delooping of the Bott periodicity homotopy equivalence 
\begin{equation}\label{eq:bott}
\Omega (U/O) \simeq \Z\times BO.
\end{equation}
We follow Giroux's proof of \eqref{eq:bott} which is implicit in \cite{Giroux1990}. This consists in a careful
study of the linear algebra underlying the theory of generating functions. The main step is a Serre fibration
statement reminiscent of the homotopy lifting property of generating functions. Our main observation is that
this Serre fibration property holds equivariantly with respect to the action of $\QQ$, eventually leading to \eqref{eq:bottequivalence}.

We consider here a symplectic vector space $E$ and a Lagrangian subspace $V$. Recall that $\QQ$
forms a unital monoid in the category of smooth manifolds (or topological spaces) under the direct sum operation
with unit given by the unique element of $\QQ_0$.

We shall presently construct a bimodule $\Lambda^V(E)$ over $\QQ$ given by the Grassmannian of Lagrangians which are stably transverse to a fixed Lagrangian $V$ in the symplectic vector space $E$. To this end, we first let $\Lambda_{n}(E)$ denote the Grassmannian of Lagrangian linear subspaces of $E \times \C^n$, which are transverse to $E \times \R^n$, and we denote their disjoint union by

\begin{equation}
  \Lambda(E) \coloneqq \bigsqcup_{n = 0}^{\infty} \Lambda_n(E).
\end{equation}
Note that, if $0$ is the trivial  vector space, then $\Lambda(0)$ is the disjoint union of the Grassmannians of Lagrangian subspaces in $\C^n$ which are transverse to $\R^n$, and the direct sum of Lagrangian subspaces makes $\Lambda(0)$ into a monoid. Moreover, $\Lambda(E)$ is a bimodule over $\Lambda(0)$, in the sense that there are commuting left and right actions of $\Lambda(0)$ on $\Lambda(E)$, which are induced by the natural isomorphisms of symplectic vector spaces
\begin{align}
  (E \oplus \C^n) \oplus \C^k & \cong E \oplus ( \C^n \oplus \C^k) \cong E \oplus \C^{n+k} \\
  \C^k \oplus (E \oplus \C^n) & \cong E \oplus (\C^k \oplus \C^n) \cong  E \oplus \C^{k+n}.
\end{align}
The reason for imposing the transversality restriction above is that $E \times \R^n$ is a coisotropic subspace of $E \times \C^n$, and that the reduction of $E \times \C^n$ along this subspace is naturally identified with $E$. We thus obtain a map
\begin{equation}
    \rho \co \Lambda(E) \to \Lambda_0(E),
\end{equation}
whose components we denote $\rho_n$. This symplectic reduction map is invariant under both the left and the right action of $\Lambda(0)$.

To relate $\Lambda(E)$ to $\QQ$, observe that we have a natural map
\begin{equation} \label{eq:graph_of_quadratic}
  \QQ_k  \to \Lambda_k(0)
\end{equation}
which maps a quadratic form $q$ to the graph of its differential in $\R^k\times i\R^k=\C^k$, where we identify $i \R^k$ with the dual space of $\R^k$ using the standard metric. This map is well-defined because this graph is transverse to $\R^k$ (it is also transverse to $i \R^k$, which will be used later). This map intertwines the direct sum of quadratic forms with the direct sum of Lagrangian subspaces, which, by the above discussion, makes $\Lambda(E)$ into a bimodule over $\QQ$.  For $q\in \QQ$ and $\phi \in \Lambda(E)$, we denote the action by
\begin{align}
\QQ  \times \Lambda(E) \times \QQ & \to \Lambda(E) \\
  (q', \phi, q) & \mapsto q' \cdot \phi \cdot q.
\end{align}

Returning to our choice of Lagrangian subspace $V$ of $E$, we define
\begin{equation}
  \Lambda^{V}_{n}(E) \subset \Lambda_{n}(E) 
\end{equation}
to be the open subsets of linear Lagrangian subspaces that are transverse to $V \oplus i \R^n$, and denote by
\begin{equation}
     \Lambda^{V}(E) \equiv \bigsqcup_{n = 0}^{\infty} \Lambda^V_n(E) \subset \Lambda(E)
\end{equation}
their disjoint union. In the case of the trivial space $E = \{ 0 \}$, the fact that every Lagrangian which is transverse to $\R^n$ and to $i \R^n$ can be uniquely expressed as the graph of a non-degenerate quadratic form implies that Equation \eqref{eq:graph_of_quadratic} factors as a map
\begin{equation} \label{eq:nearby_twisted:1}
   \QQ_k  \to \Lambda^0_k(0) \to \Lambda_k(0),
\end{equation}
where the first map is a homotopy equivalence of monoids. Since the direct sum of Lagrangians preserves transversality, we conclude that the inclusion of $\Lambda^{V}(E)$ in $ \Lambda(E)$ is an inclusion of $\QQ$-bimodules, i.e. that $\Lambda^{V}(E) $ inherits commuting right and left actions of $\QQ$:
\begin{equation}
\QQ  \times \Lambda^V(E) \times \QQ  \to \Lambda^V(E).
\end{equation}

Restricting the symplectic reduction map $\rho$ to Lagrangians stably transverse to $V$, we obtain a map that we denote
\begin{equation} \label{eq:symplectic_reduction}
    \rho:\Lambda^V(E)\to \Lambda_{0}(E).
\end{equation}
\begin{lem}\label{lem:fiber}
If a Lagrangian $H\in \Lambda_0(E)$ is transverse to $V$, the $\QQ$-equivariant map $\QQ\to \rho^{-1}(H)$
defined by $q\mapsto q \cdot H$ is a homotopy equivalence of right $\QQ$-spaces.
\end{lem}
\begin{proof}
  As the first map in Equation~(\ref{eq:nearby_twisted:1}) is a homotopy equivalence of monoids we may prove the lemma for $\Lambda^0(0)$ instead of $\QQ$.
  
  The symplectic form on $E$ identifies $V$ with the linear dual $H^*$ and $E$ with $H\oplus H^*$, hence an element
$\phi\in \Lambda_n^V(E)$ can be represented uniquely as a quadratic form $\phi$
on $H\times \R^n$. We set $\phi_t(h,v)=\phi(t\cdot h,v)$. If $\rho(\phi)=H$,
then for all $t\in[0,1]$, $\rho(\phi_t)=H$ and $\phi_0(h,v)=\phi(0,v)$, which represents an element in the image of the action $\Lambda^0(0) \to \rho^{-1}(H)$. Observe that if $q'\in \Lambda^0(0)$ then $(\phi\cdot q')_t=\phi_t\cdot q'$, hence we have constructed the required $\Lambda^0(0)$-equivariant deformation retraction onto a copy of $\Lambda^0(0)$.
\end{proof}

The theory of generating functions starts in earnest with the result that the map $\rho$ from Equation \eqref{eq:symplectic_reduction} behaves like a fibration after sufficient stabilizations by the quadratic form $h$. We observe here that this fibration property also holds $\QQ$-equivariantly. One can state such a result in terms of a homotopy lifting property for maps of compact $CW$-complexes into $\Lambda_{0}(E)$, but since our desired application requires this lift to be smooth, we consider a compact smooth manifold $L$ parametrizing smooth paths of Lagrangians in $E$, i.e. equipped with a smooth map
\begin{equation}
   g\colon L\times [0,1]\to \Lambda_{0}(E).
\end{equation}
Since the projection from $\Lambda^V(E)$ to $ \Lambda_{0}(E)$ is a  smooth submersion, the pullback is then a smooth manifold over $L \times [0,1]$, which we denote $ g^*\Lambda^V(E)$. Letting $g_0$ denote the restriction to $L \times \{0\}$, we have the following result, which asserts that the failure of $\rho$ to be a fibration is resolved by stabilization:
\begin{lem}\label{lem:chekanovlinear}
If $N\in \N$ is sufficiently large, there exists a smooth map
\begin{equation}
\Theta\colon g_0^*\Lambda^V(E)\times [0,1]\to g^*\Lambda^V(E)     
\end{equation}
fibered over $L\times [0,1]$ such that
\begin{itemize}
\item$\Theta(\phi,0)=h^N\cdot \phi$,
\item $\Theta$ is (right) $\QQ$-equivariant, in the sense that, for all $q\in \QQ$,
  $\phi \in g_0^*\Lambda^V(E)$ and $t\in[0,1]$ we have
  \begin{equation}
 \Theta(\phi \cdot q,t)=\Theta(\phi,t)\cdot q.  
  \end{equation}  
\end{itemize}
\end{lem}
\begin{proof}
  For concreteness, we choose identifications $E=\C^m$ and $V=i\R^m$, and our first step is to use the fact that the map $\Sp_{2m} \to \Lambda_0(\C^m)$ is a smooth fiber bundle to pick a smooth map $\theta\colon L\times [0,1]\to \Sp_{2m}$ which is the identity over $L \times \{0\}$, and such that $\theta_t(g_0)=g_t$ for all $t\in[0,1]$. We now explain the construction if this path $\theta_t$ remains in a neighbourhood of the identity to be specified below.

To that end, let $\theta\in \Sp_{2m}$ be any symplectic map which is close to the identity. We first associate to $\theta$ a quadratic form $S^\theta\colon \R^{2m}\to \R$. We have a symplectomorphism between $\C^m \times \overline{\C}^m$ and $T^*\R^{2m}$ given by the map
\begin{align}
  (x_1+iy_1,y_2+ix_2) \mapsto (x_1,x_2,y_1,y_2),
\end{align}
where $(x_1,x_2)$ denotes the coordinates on the base $ \R^{2m}$ of the cotangent bundle. The graph of $\theta^{-1}$ in $\C^m \times \overline{\C}^m$ is Lagrangian and graphical over the base in $T^*\R^{2m}$ (this graphicality condition is the precise closeness between $\theta$ and the identity which we need). So there is a unique quadratic form $S^\theta(x_1,x_2)$ so that $dS^\theta$ is the image of this graph, i.e. so that for every pair $(x_1,x_2)$ we have
\begin{equation} \label{eq:graph_theta_graph_dS}
  \theta^{-1}\left(x_1 + i  \frac{\del S^{\theta}}{\del x_1}(x_1,x_2) \right) = \frac{\del S^{\theta}}{\del x_2}(x_1,x_2)+ix_2.
\end{equation}
Note that, for the identity of the group $\Sp_{2m}$, we have
\begin{align} \label{eq:S-of_identity}
  S^\id(x_1,x_2) = \langle x_1 , x_2 \rangle
\end{align}
where $\langle - , - \rangle$ denotes the standard inner product on $\R^m$.

Our assumptions on $ \Lambda^V_n(\C^m)$ imply that any element $\phi$ of this space is similarly the graph of a quadratic form on $\R^{m+n}$, which we write as $\phi(x,v)$ with $(x,v)\in \R^m\times \R^n$. Now consider the quadratic form on $ \R^{3m+n}$ 
\begin{equation} \label{eq:formula_quadratic_form_phi'}
  \phi'(X_1,X_2,X_3,v) = \phi(X_3,v) + S^{\theta}(X_1,X_2) - \langle X_3 , X_2 \rangle,
\end{equation}
where the coordinate $(X_1, X_2, X_3)$ are on $\R^{3m}$. 
 The critical points of $\phi'$ over a fixed $X_1\in \R^m$ are the points $(X_2,X_3,v)$ which satisfy
\begin{align}
  0 = \frac{\del \phi}{\del v}, \qquad X_2 =  \frac{\del \phi}{\del X_3}, \qquad X_3 =   \frac{\del S^{\theta}}{\del X_2},
\end{align}
and at these points we have $\frac{\del \phi'}{\del X_1} = \frac{\del S^{\theta}}{\del X_1}$. Plugging into Equation \eqref{eq:graph_theta_graph_dS}, and identifying $x = X_1$,  we conclude that
$\rho(\phi') = \theta(\rho(\phi))$.

When $\theta = \id$, the above $\phi'$ is not obtained by left multiplication by $h^m$. We fix this by making a \emph{linear} change of variables
\begin{equation}
 x_1 = X_1,\quad x_2 = X_2 - f(X_1,X_3,  v), \quad x_3= X_1 - X_3.
\end{equation}
where $f\colon \R^{2m} \times  \R^n \to \R^m$ is a linear function solving
\begin{align} \label{eq:twisted_change_of_coordinates_X_2}
  \langle f(X_1, X_3, v),X_1 -X_3 \rangle = \phi(X_1,v) -  \phi(X_3,v) ,
\end{align}
which is solved explicitly using the symmetric matrix representing $\phi$. Performing the above substitution in Equation \eqref{eq:formula_quadratic_form_phi'}, we define
\begin{equation}
  \phi^\theta(x_1,x_2,x_3,v)=\phi'(X_1,X_2,X_3,v),
\end{equation}
and compute that, for $\theta = \id$, we have,
\begin{align}
  \phi^\id(x_1,x_2,x_3,v) & = \phi(X_3,v) + \langle X_1 - X_3, X_2 \rangle \\
  & = \phi(X_3,v) + \langle x_3, x_2 + f(X_1, X_3, v) \rangle \\
  & = \phi(x_1, v) + \langle x_3, x_2 \rangle.
\end{align}
Up to reordering the co-ordinates labelled $x_2$ and $x_3$ (and rescaling $x_2$ by a factor of $2$, recall $h(x,y)=2xy$ because
we stick to non-degenerate quadratic forms with eigenvalues $\pm 1$), we thus conclude that
\begin{equation}
  \phi^\id=h^m\cdot\phi
\end{equation}
as desired. In addition, it is easy to see that the remaining properties
\begin{align}
  \rho(\phi^\theta) & =\theta(\rho(\phi)) \\
  (\phi\cdot q)^\theta &=\phi^\theta \cdot q
\end{align}
still hold. We denote $\Theta(\phi)=\phi^\theta$.

The above construction varies smoothly in the choice of map $\theta$, and is well-defined as long as it remains in a neighbourhood of the identity, yielding a smooth path $\Theta_t$. In general, we decompose $\theta_t$ as a product
\begin{equation}
 \theta_t=\theta_t^K  \circ \cdots \circ \theta_t^1 
\end{equation}
of paths $\{ \theta_t^i\}_{i=1}^K$ in $\Sp_{2m}$ each of which stays close enough to the identity. We then define $\Theta_t$ as the composition
\begin{equation}
  \Theta_t \equiv \Theta^{K}_{t} \circ \cdots  \circ \Theta^{1}_{t} 
\end{equation}
 which has the desired properties with $N=mK$.
\end{proof}

We would like to express the above result as a Serre fibration statement and for that
we need to pass to a limit where the stabilization $\phi\mapsto h^N\cdot \phi$ in
Lemma~\ref{lem:chekanovlinear} becomes the identity. So we define:
\begin{align}
\QQ_\infty &= \colim (\QQ \xrightarrow{h \cdot} \QQ
\xrightarrow{h  \cdot} \QQ \to\cdots) \\
\Lambda^V_\infty(E)  &= \colim ( \Lambda^V(E) \xrightarrow{h \cdot}
\Lambda^V(E) \xrightarrow{h \cdot} \Lambda^V(E) \to\cdots).
\end{align}

We now give the Serre fibration statement which essentially goes back to Giroux \cite{Giroux1990} and Latour \cite{Latour1991}
and its corresponding statement in terms of the (one-sided) bar construction applied to the monoid $\QQ $ and the module $\Lambda_\infty^V(E)$, as discussed in Appendix~\ref{app:monoids-and-bundles}. To clarify the meaning of the next result, it may be useful to recall that the bar construction $|B(\Lambda_\infty^V(E), \QQ)|$ is a model for the (homotopy) quotient of $\Lambda_\infty^V(E)$ by the action of $\QQ$, so that the second statement below asserts that this quotient is equivalent to $\Lambda_0(E) $.
\begin{thm}\label{thm:giroux} 
  The reduction map $\Lambda_\infty^V(E)\to \Lambda_0(E)$ is a Serre fibration with fiber homotopy equivalent as right $\QQ$-spaces to $\QQ_\infty$.

  The map $|B(\Lambda_\infty^V(E), \QQ)|\to \Lambda_0(E)$ is a Serre fibration with
  contractible fibers, hence a homotopy equivalence.
\end{thm}
\begin{proof}
Given a map $g\colon D^n\times [0,1]\to\Lambda_0(E)$, let $\phi_0\colon D^n\to \Lambda^V_\infty(E)$ be a lift of $g_0=g|_{D^n\times\{0\}}$.
Lemma~\ref{lem:chekanovlinear} provides a map $\Theta\colon g_0^*\Lambda^V(E)\times [0,1]\to g^*\Lambda^V(E)$.
Since $D^n$ is compact, we can represent $\phi_0$ as a map
$\tilde\phi_0\colon D^n\to \Lambda^V(E)$ at some finite step of the colimit $\Lambda^V_\infty(E)$.
Then $\phi_t=\Theta_t \circ\phi_0$ composed with the map $\Lambda^V(E)\to \Lambda^V_\infty(E)$ is the required lift of $g$.
Lemma~\ref{lem:fiber} implies the claim about the fiber. 

The second assertion does not formally follow from the first so we have to rephrase its proof slightly.
We start again with a map $g:D^n\times [0,1]\to\Lambda_0(E)$ and the  associated operator $\Theta$ from Lemma~\ref{lem:chekanovlinear}
but consider instead a lift $\psi_0:D^n\to|B(\Lambda_\infty^V(E), \QQ)|$ of $g_0$.
Since the bar construction and the geometric realization commute with colimits,
  we have 
\begin{equation}|B(\Lambda^V_\infty(E),\QQ)| = \colim(|B(\Lambda^V(E),\QQ)| \xrightarrow{h\cdot} |B(\Lambda^V(E),\QQ)| \xrightarrow{h\cdot} \cdots).\end{equation}
Since $D^n$ is compact, we can represent $\psi_0$ by a map $\psi'_0\colon D^n\to|B(\Lambda^V(E),\QQ)|$ at a finite step. That this
map lifts $g_0$ means that it is a section of the pull back $g_0^*|B(\Lambda^V(E),\QQ)| \to D^n$.
The map $\Theta_t$ is $\QQ$-equivariant and hence induces a map $\Theta'_t\colon g_0^*|B(\Lambda^V(E),\QQ)|\to g^*|B(\Lambda^V(E),\QQ)|$.
We set $\psi'_t=\Theta'_t\circ\psi'_0$ and compose it back with the map to the colimit $|B(\Lambda_\infty^V(E), \QQ)|$
to get the required lift $\psi_t$. From the first part of the statement, we know that the fiber is homotopy equivalent to $|B(\QQ_\infty,\QQ)|$.
But we have 
\begin{equation}|B(\QQ_\infty,\QQ)|=\colim(|B(\QQ,\QQ)| \xrightarrow{h\cdot}   |B(\QQ,\QQ)| \xrightarrow{h\cdot}  \dots),\end{equation}
$|B(\QQ,\QQ)|$ is contractible according to Lemma~\ref{lem:BQQ}, and each arrow in the above colimit is a cofibration, so we conclude that this fiber is contractible.  
\end{proof}

\begin{rem}\label{rem:ZxZxBO}
Writing the space $\QQ_\infty$, as the quotient
of $\QQ\times \N$ by the equivalence relation generated by $(q,i)\sim (h\cdot q,i+1)$, we have two maps $\QQ_\infty\to \Z$, respectively corresponding to the dimension and the signature:
\begin{align}
  (q,i) &\mapsto \dim(q) - 2i,\\
  (q,i) &\mapsto \sgn(q)=\coind(q)-\ind(q).
\end{align}
These maps descend to $\QQ_\infty$ since $\dim(h)=2$ and $\sgn(h)=0$.
Finally, by taking the sum of the negative eigenspaces we also have a map 
 \begin{equation}
      E^- \colon \QQ_{\infty} \to BO,
 \end{equation}
This altogether yields a homotopy equivalence
   \begin{equation}
\QQ_\infty \simeq \Z \times \Z \times BO.
   \end{equation}
\end{rem}

We would like now to pass to a second limit when we replace $(E,V)\simeq (\C^m, i\R^m)$ by $(\C^\infty, i\R^\infty)$.
This is motivated by the definition of stable Gauss map which takes values in $\Lambda_0(\C^\infty)$ and also by the following lemma.

\begin{lem}\label{lem:connectedness}
For all $n,m\in \N$, the space $\Lambda_n^{i\R^m}(\C^m)$ is $(m-1)$-connected. 
\end{lem}
\begin{proof}
The space $\Lambda^{i\R^{m+n}}_0(\C^{m+n})$ is contractible and the set of $\phi\in \Lambda^{i\R^{m+n}}_0(\C^{m+n})$
such that $\phi$ is not transverse to $\C^m\times \R^n$ has codimension at least $m+1$, hence
$\Lambda^{i\R^m}_n(\C^m)$ is $(m-1)$-connected.
\end{proof}

We have already defined $\Lambda_0(\C^\infty)$ as

\begin{equation}\Lambda_0(\C^\infty)=\colim(\Lambda_0(\C)\xrightarrow{\R\oplus \cdot} \Lambda_0(\C^2) \to \cdots ),\end{equation}
and now introduce
\begin{equation}\Lambda_\infty^{i\R^\infty}(\C^\infty)=\colim(\Lambda_\infty^{i\R}(\C)\xrightarrow{\R\oplus \cdot} \Lambda_\infty^{i\R^2}(\C^2) \to \cdots ).\end{equation}
Note that the left action used to define $\Lambda_\infty^V(E)$ commutes with the maps in this limit,
and both commute with the right action. Hence $\QQ$ still acts on the right on
$\Lambda_\infty^{i\R^\infty}(\C^\infty)$ and symplectic reduction gives a $\QQ$-invariant map
\begin{equation}
\Lambda_\infty^{i\R^\infty}(\C^\infty) \to \Lambda_0(\C^\infty).
\end{equation}
Similarly as in Remark~\ref{rem:ZxZxBO}, the map $\Lambda^{i\R^m}(\C^m)\times \N \to \Z$ which takes $(\phi,i)$
to $\dim(\phi)-m-2i$ descends to a map
\begin{equation}\label{eq:Zcomponent}
\Lambda_\infty^{i\R^\infty}(\C^\infty) \to \Z .
\end{equation}
Indeed,  $\Lambda_\infty^{i\R^\infty}(\C^\infty)$ can alternatively be described as the quotient 
\[(\coprod_{m\in \N}\Lambda^{i\R^m}(\C^m)\times \N)/\sim\]
where $(\phi,i)\sim (h\cdot \phi, i+1)$ and $(\phi,i)\sim (\R\oplus \phi, i)$. Note also that the map \eqref{eq:Zcomponent}
is right $\QQ$-equivariant. We now obtain the final result of this subsection.

\begin{thm}\label{thm:girouxbott}
  The reduction map $\Lambda_\infty^{i\R^\infty}(\C^\infty)\to \Lambda_0(\C^\infty)$ is a Serre fibration with fiber homotopy equivalent to $\QQ_\infty$
and the induced map 
\[|B(\Lambda_\infty^{i\R^\infty}(\C^\infty), \QQ)|\to \Lambda_0(\C^\infty)\]
is a Serre fibration with contractible fibers, hence a homotopy equivalence.

The map $\Lambda_\infty^{i\R^\infty}(\C^\infty) \to \Z$ from Equation~\eqref{eq:Zcomponent} is a homotopy equivalence of right $\QQ$-spaces so the natural map
\[|B(\Lambda_\infty^{i\R^\infty}(\C^\infty), \QQ)|\to|B(\Z,\QQ)|\]
is a homotopy equivalence.

In particular, we get a homotopy equivalence  
  \begin{equation}\label{eq:BQ}
    |B(\Z,\QQ)|\simeq \Lambda_0(\C^\infty)
  \end{equation} which is canonical up to homotopy.
\end{thm}

\begin{proof}
The statements about Serre fibrations follow from Theorem~\ref{thm:giroux} using the fact that a colimit of Serre fibrations is a Serre fibration.

The fact that $\Lambda_\infty^{i\R^\infty}(\C^\infty)\to \Z$ is a homotopy equivalence of right $\QQ$-spaces follows from Lemma~\ref{lem:connectedness} which implies that $\Lambda_n^{i\R^\infty}(\C^\infty)$ is contractible for all $n\in \N$.
\end{proof}

\begin{rem}\label{rem:bott}
In view of Remark~\ref{rem:ZxZxBO}, Theorem~\ref{thm:girouxbott} and an argument similar to \cite[Lemma D.1]{hatcher_short_2014}) gives
a fibration sequence
\[ \Z \times \Z \times BO \to \Z \to U/O\]
which implies the Bott periodicity homotopy equivalence $\Omega (U/O) \simeq \Z\times BO$.
\end{rem}

When dealing with generating functions, it is more natural to have $\N$ instead of $\Z$ (since the domain of a generating function
has non-negative dimension) and also useful to translate globally (by stabilization), so we record the following lemma.

\begin{lem}\label{lem:translation}
The map $|B(\N,\QQ)|\to |B(\Z,\QQ)|$ induced by the inclusion $\N \to \Z$ is a homotopy equivalence.

The map $|B(\Z,\QQ)|\to |B(\Z,\QQ)|$ induced by the map $\Z\to\Z$, $x \mapsto x+1$,
is homotopic to the identity.
\end{lem}
\begin{proof}
  Let $X=|B(\N,\QQ)|$ and consider a $\QQ$-twisted map $(X_{i\in I},n_i,q_{ij})$ from $X$ to $\N$ which using Proposition~\ref{prop:QF-bundles} corresponds to the identity on $X$. Fix any $q\in \QQ_1$. Let $J=I\sqcup I$ where we denote $i$ in the first copy by $i'$ and in the second copy by $i''$. We extend the induced order on $J$ by defining $i'<j''$ for all $i,j\in I$. Consider the open cover of $X\times [0,1]$ given by sets $X_{i'} = X_i\times [0,\tfrac23)$ and $X_{i''}=X_i\times (\tfrac13,1]$. We define a $\QQ$-twisted map $X\times I$ to $\N$ by letting $n_{i'}=n_i$ and $n_{i''}=n_i+1$, $q_{i'j'}=q_{ij}$, $q_{i''j''}=q_{ij}$ and $q_{i'j''} = q\oplus q_{ij}$. In particular $q_{i',i''}=q$. This shows that the $\QQ$-twisted map given by $(X_i,n_i+1,q_{ij})$ is equivalent (homotopic) to the identity.

  As $B(\Z,\QQ) = \colim(B(\N,\QQ) \xrightarrow{+1} B(\N,\QQ) \xrightarrow{+1} B(\N,\QQ) \cdots )$ all claims follow.
\end{proof}


\subsection{Proof of Theorem~\ref{thmintro:twistedgirouxlatour}}\label{subsec:existencetgf}

In this subsection we prove Theorem~\ref{thm:twistedgirouxlatour} by adapting Giroux's proof of Theorem~\ref{thm:girouxlatour}
to the twisted setting.

\subsubsection{``only if'' part}

Given a generating function $(n,U,f,\psi)$ for a Legendrian immersion $\varphi\times z\colon L\to J^1 M$, at each point of $\psi(L)$
the tangent space to the graph of $df$ is a Lagrangian subspace $\phi$ of $E\times \C^k$ (recall $E=\varphi^*(TT^* M)$ and $V \subset E$
is the vertical Lagrangian subbundle) with the following properties:

\begin{itemize}
\item it is transverse to $E\times \R^k$ (by the transversality condition in Definition~\ref{dfn:gf}),
\item it is transverse to $V\times i\R^k$ (because it is a graph),
\item the reduction of $\phi$, i.e. projection of $\phi \cap (E\times \R^k)$ to $E$, is equal to the Gauss section $G$.
\end{itemize}
Hence with the notations of Subsection~\ref{subsec:girouxbott}, extended verbatim to the present bundle situation over $L$, we have
a section

\begin{equation}
\phi\colon L\to \Lambda^V(E)
\end{equation}
which lifts the Gauss section $G$ under the reduction map $\rho\colon\Lambda^V(E)\to \Lambda_0(E)$.
If we pick a complement $H'$ as in Definition~\ref{dfn:stablegaussmap},
we trivialize the situation and obtain a map 
\[H'\oplus \phi\colon L\to \Lambda^{i\R^m}(\C^m)\]
which lifts the Gauss map $L\to \Lambda_0(\C^m)$. Composing with the map 
\[\Lambda^{i\R^m}(\C^m) \to \Lambda_\infty^{i\R^m}(\C^m)\]
(inclusion, say, at the first place in the colimit) and with the map 
\[\Lambda_\infty^{i\R^m}(\C^m) \to \Lambda_\infty^{i\R^\infty}(\C^\infty),\]
we get a lift of the stable Gauss map $g_\varphi\colon L\to \Lambda_0(\C^\infty)$ to $\Lambda_\infty^{i\R^\infty}(\C^\infty)$.
But Theorem~\ref{thm:girouxbott} implies that the reduction map $\Lambda_\infty^{i\R^\infty}(\C^\infty)\to \Lambda_0(\C^\infty)$
is nullhomotopic (since $\Lambda_\infty^{i\R^\infty}(\C^\infty)$ is homotopy equivalent to $\Z$ and $\Lambda_0(\C^\infty)$
is connected) and hence so is $g_\varphi$, proving the ``only if'' part of Theorem~\ref{thm:girouxlatour}.

In order to extend this result to the twisted case, observe that a twisted generating function gives a directed open cover $(M_i)_{i\in I}$, together with
 a map $n_i\colon M_i \to \N$ and a lift $\phi_i\colon L_i=\pi^{-1}(M_i)\to \Lambda_{n_i}^V(E)$
of the Gauss section $G$ for each $i$, and for all $i<j$ a map $q'_{ij}\colon M_{ij}\to \QQ$
such that $\phi_i\oplus q_{ij}=\phi_j$ on $L_{ij}$ where $q_{ij}=q'_{ij}\circ \pi \colon L_{ij}\to \QQ$.
We use the notation $MV(M_\bullet)$ from Appendix \ref{sec:directed-open-covers} and the equivalence between
twisted sections and  simplicial maps $MV(M_\bullet) \to B(-,\QQ)$ from Appendix \ref{sec:app_associated_bundles}.
After trivializing with $H'\oplus H\simeq L\times \R^m$ and passing to the limit as $m\to \infty$, the situation
can be summarized in the following commutative diagram of simplicial spaces:

\begin{equation}
\begin{tikzcd}
MV(M_\bullet) \ar[r, "n_\bullet"] & B(\N,\QQ) \ar[r,"\sim"] & B(\Z,\QQ) \\
MV(L_\bullet) \ar[u,"\pi"] \ar[r, "H'\oplus \phi_\bullet"] \ar[d,"\sim"] & B(\Lambda^{i\R^\infty}(\C^\infty),\QQ) \ar[r] \ar[u] \ar[d] & B(\Lambda_\infty^{i\R^\infty}(\C^\infty),\QQ) \ar[u,"\sim"]\ar[ld,"\sim"] \\
L \ar[r,"g_\varphi"] & \Lambda_0(\C^\infty)&.\\
\end{tikzcd}
\end{equation}
Here, the two spaces in the bottom row are the constant simplicial spaces, and the symbol $\sim$ indicates that the given arrow is a homotopy equivalence after passing to geometric realizations. 

After applying the geometric realization functor, we get the following homotopy commutative diagram which proves the ``only if'' part of Theorem~\ref{thm:twistedgirouxlatour}, namely that $h\circ\pi$ is homotopic to $g_\varphi$:

\begin{equation}
\begin{tikzcd}
M  \ar[rd,"h"] &  {}\\ 
{|MV(M_\bullet)|}\ar[u,"\sim"] \ar[r] & {|B(\Z,\QQ)|} \\
{|MV(L_\bullet)|} \ar[u,"\pi"] \ar[r]\ar[d,"\sim"] & {|B(\Lambda_\infty^{i\R^\infty}(\C^\infty),\QQ)|} \ar[u,"\sim"] \ar[d,"\sim"] \\
L \ar[r,"g_\varphi"] & \Lambda_0(\C^\infty).\\
\end{tikzcd}
\end{equation}

\subsubsection{``if'' part}

We now turn to the ``if'' part of Theorem~\ref{thm:twistedgirouxlatour} which is somewhat more involved.
We start by picking smooth trivializations $E\simeq H \oplus H^*$, $H'\oplus H \simeq L \times \R^m$
as in Definition~\ref{dfn:stablegaussmap} with $m$ large enough so that the map 
\[h\colon M\to \Lambda_0(\C^\infty)\]
is represented up to homotopy by
a smooth map
\[h' \colon M\to \Lambda_0(\C^m).\]
Pick also a smooth non-degenerate quadratic form $q_{H'}$ on $H'$ and consider the associated smooth Gauss map
\[g'=q_{H'}\oplus G \colon L\to \Lambda_0(\C^m)\]
(this is homotopic to $H'\oplus G$ so is also a valid choice for Definition~\ref{dfn:stablegaussmap}).
By making $m$ larger we can assume that $g'$ is homotopic to $h'\circ \pi$.
We insist here on smoothness since in the end we want to construct an actual smooth twisted generating
function $(f_i,q_{ij})$.

\paragraph{Step 1 : construction of an infinitesimal twisted generating function}

In this first step we construct what will eventually be the tangent spaces $\phi_i$ to the graph of the $1$-forms $df_i$
and the quadratic forms $q_{ij}$.
In view of Theorem~\ref{thm:giroux}, symplectic reduction gives a homotopy equivalence
\[\rho\colon |B(\Lambda^{i\R^m}_\infty(\C^m),\QQ)|\to \Lambda_0(\C^m).\]
Composing $h'$ with a homotopy inverse of $\rho$ gives a map $M\to|B(\Lambda^{i\R^m}_\infty(\C^m),\QQ)|$
which we can represent at a finite step of the colimit as a map
\[M\to |B(\Lambda^{i\R^m}(\C^m),\QQ)|\]
to which we apply Proposition~\ref{prop:QF-bundles}. This yields a $\QQ$-twisted map
from $M$ to $\Lambda^{i\R^m}(\C^m)$,
namely a directed open cover $(M_i)_{i \in I}$ of $M$ and a simplicial map
\[(\phi,q)\colon MV(M_\bullet) \to B(\Lambda^{i\R^m}(\C^m),\QQ)\]
which covers a map $M\to \Lambda_0(\C^m)$ homotopic to $h'$. By Lemma~\ref{App:smoothing} we may assume that the functions are smooth.
We may also assume that, for each $i$, $M_i$ is connected and so $\phi_i$ takes values in $\Lambda_{n_i}^{i\R^m}(\C^m)$ for some integer $n_i$.

We pull back this data to $L$ under $\pi\colon L\to M$ to obtain a directed cover $L_i=\pi^{-1}(M_i)$
and (smooth) simplicial map 
\[MV(L_\bullet) \to B(\Lambda^{i\R^m}(\C^m),\QQ)\]
given by $\phi'_i=\phi_i\circ \pi\colon L_i\to \Lambda_{n_i}^{i\R^m}(\C^m)$ and $q'_{ij}=q_{ij}\circ \pi\colon L_{ij}\to \QQ$. The issue is
that $\phi'_i$ is not exactly a lift of the Gauss map $g'\colon L\to \Lambda_0(\C^m)$. However the maps $\rho\circ\phi'_i$
glue together into a map $g'_0\colon L\to \Lambda_0(\C^m)$ which is homotopic to $g'$. We can thus apply Lemma~\ref{lem:chekanovlinear}
to such a smooth homotopy $g'_t$, $t\in[0,1]$, with $g'_1=g'$ to obtain a smooth family of maps 
\[\Theta_t\colon (g'_0)^*\Lambda^{i\R^m}(\C^m)\to (g'_t)^*\Lambda^{i\R^m}(\C^m).\]
We set $\phi''_i=\Theta_1\circ \phi_i'\colon L_i \to \Lambda_{n_i+2N}^{i\R^m}(\C^m)$. This still satisfies $\phi''_i\oplus q_{ij}=\phi''_j$ due to the equivariance property of $\Theta_1$ and is thus a $\QQ$-twisted lift of $g'\colon L\to \Lambda_0(\C^m)$ to $\Lambda^{i\R^m}(\C^m)$.

Finally we undo the trivialization of $E$ as follows. Recall $E'\oplus E=L\times \C^m$ where $E'=H'\oplus H'^*$ and consider the permutation map
\[a_n\colon (E\oplus E'\oplus E) \times \C^n \to (E\oplus E' \oplus E)\times \C^n\]
given by $a_n(e_1,e',e_2,x)=(e_2,e',e_1,x)$. Let us also pick a non-degenerate quadratic form $q_H$ on $H$ and consider the $\QQ$-equivariant map 
\begin{equation}S\colon \coprod_{n\in \N}\Lambda^{i\R^m}_n(\C^m)\times X\to \coprod_{n\in \N}\Lambda_{m+n}^V(E)\end{equation}
defined by $\phi\in \Lambda_n^{i\R^m}(\C^m) \mapsto a_n(q_H \oplus \phi)$. We can check using the non-degeneracy of $q_H$ and $q_{H'}$
that $S$ actually lands in $\Lambda_{m+n}^V(E)$ and that $(S\circ\phi''_i,q'_{ij})$ is a $\QQ$-twisted lift of $G$ to $\Lambda^V(E)$.
For simplicity, let us rename $S\circ\phi_i''$ as $\phi_i$ from now on.

\paragraph{Step 2 : compatibility with a fixed embedding}
We now choose an embedding $\psi=\pi\times \psi'\colon L\to M\times \R^k$ covering $\pi\colon L\to M$ under the first projection.
Together with the Legendrian immersion $\varphi\times z$ we obtain an isotropic embedding of $L$ in the contact manifold $J^1(M\times \R^k)$:
\[\theta=\varphi\times\psi'\times z \colon L\to T^*M\times \C^k\times \R \simeq J^1(M\times \R^k).\]
A generating function $(\psi,U,f)$ for $\varphi\times z$ is nothing more (via its $1$-jet graph)
than a Legendrian submanifold of $J^1(M\times \R^k)$ which is graphical over $U$, and meets $T^*M\times \R^k\times\R$ transversely along $\theta(L)$. In this part, we construct the tangent field of this Legendrian submanifold. 

Given a lift $\phi$ of the Gauss map $G$ to $\Lambda_k^V(E)$, we have a monomorphism
\[u_\phi\colon TL \simeq\phi\cap (E\times \R^k) \to E/V \times \R^k.\]

Given $\psi'\colon L\to \R^k$, we can identify the vector bundle $E/V\times \R^k$ over $L$ with the
pullback of $TM\times T \R^k$ under $\pi\times \psi' \colon L\to M\times \R^k$. In this case $u_\phi$
is viewed as a monomorphism
\[u_\phi \colon TL \to TM \times T\R^k\]
covering $\pi\times \psi'$. In the same way, the pullback of $d(\pi\times \psi')$ is a monomorphism
$TL \to E/V\times \R^k$ above $L$. The Legendrian lift of $\phi$ is a field of Legendrian planes
along $\theta(L)$ and this field is tangent to $\theta(L)$ if and only if $\phi$ contains the image of
$d(\varphi\times \psi')$. Now we claim that this condition is equivalent to the equality
$u_\phi=d(\pi\times \psi')$. Indeed the inverse of the projection $\phi\cap (TT^*M \times T\R^k)\to TL$
is of the form $d\varphi\times v_\phi$ (because $\phi$ lifts the Gauss map), and $u_\phi=d(\pi\times \psi')$ is equivalent to $v_\phi=d\psi'$
which is further equivalent to $\phi\cap (TT^*M \times T\R^k)=\im(d\varphi\times d \psi')$ and to
$\im(d\varphi\times d\psi')\subset \phi$ for dimension reasons. We sum up this discussion in the following diagram:

\begin{equation*}
\begin{tikzcd}
& TT^* M \times T \R^k \ar[rd] & \\
TL \ar[ur,"d\varphi\times d\psi'"] \ar[dr,"d\varphi\times v_\phi"'] \ar[rr, bend right =10, "u_\phi"]\ar[rr, bend left =10, "d(\pi\times \psi')"] & & TM \times T \R^k\\
& \phi\cap (TT^*M \times T\R^k). \ar[ur]&
\end{tikzcd}
\end{equation*}

Let us call a lift $\phi$ of $G$ to $\Lambda_{k+n}^V(E)$ \emph{compatible} with $\psi$ if the monomorphism $u_\phi$ coincides with $d\psi\times 0\colon G\to E/V\times \R^k\times  \R^n$. Our goal is to make all $\phi_i$ compatible with $\psi$ at the same time. We first
stabilize all $\phi_i$, replacing them by $q\cdot \phi_i \in \Lambda_{k+n_i}^V(E)$ for some non-degenerate quadratic form $q$ on $\R^k$.
Then we will achieve the compatibility condition using the action of fiberwise diffeomorphisms at the linear algebra level,
following the procedure used by Giroux in his proof of Theorem~\ref{thm:girouxlatour}.
Namely consider the subgroup $P^m$ of $GL(E/V\times \R^m)$ consisting of linear automorphisms of the vector bundle $E/V\times \R^m$ which preserve the projection to $E/V$, i.e. those that have the form 
\[A=\begin{pmatrix} \id_{E/V} & 0 \\ a & b\end{pmatrix},\]
with $a\colon E/V\to \R^m$ and $b\in GL(\R^m)$. We consider here $P^m$ as a bundle over $L$. 
There is a natural symplectic action of $P^m$ on $E\times \C^m=E\times \R^m\times i\R^m$ as follows:
  \begin{equation}\label{eq:def tilde A}
    A\mapsto \tilde A=\begin{pmatrix} \id_E & 0 & -{}^t a {}^t b^{-1} \\ a \pi & b & 0 \\ 0 & 0  & {}^t b^{-1} \end{pmatrix},
  \end{equation}
  where $\pi\colon E\to E/V$ is the projection and
  ${}^t a\colon  i\R^m=(\R^m)^* \to (E/V)^*\simeq V\subset E$. If we fix a Lagrangian subspace $H$ of $E$ transverse to $V$, then
we have $E/V\simeq H$, $E\times \C^m \simeq (H \times \R^m)\times (H \times \R^m)^*$ and the above action is the natural symplectic lift
of $GL(H\times \R^m)$, but it turns out that the lift is independent of the choice of $H$.
Note that $\widetilde{A}$ preserves $V\oplus i\R^m$, $E\oplus \R^m$ and therefore the group $P^m$ acts on $\Lambda_m^V(E)$.
Moreover $P^m$ preserves the symplectic reduction: for all $A\in P^m$ and $\phi \in \Lambda_m^V(E)$, we have 
\begin{equation}
\rho(\widetilde{A}(\phi))=\rho(\phi).
\end{equation}
Also the monomorphism $u_\phi\colon\rho(\phi)\to E/V\times \R^m$ is altered by composition with $A$:
\begin{equation}
u_{\widetilde{A}(\phi)}=A\circ u_\phi.
\end{equation}
We denote $\Pi\colon G\to E/V$ the differential of $\pi\colon L\to M$ and write 
\[u_\phi=\Pi\times v_\phi\colon G\to E/V\times \R^m,\]
and
\[d\psi=\Pi\times v\colon G\to E/V \times\R^k\]
the differential of the embedding $\psi$. For each $i$, we have $\phi_i\in\Lambda_{k+n_i}^V(E)$, $u_{\phi_i}=\Pi\times 0 \times v_{\phi_i}$
and, setting $\psi_i=\psi\times 0\colon L\to \R^k\times \R^{n_i}$, $d\psi_i=\Pi\times v \times 0$. We will construct a family of sections
$A_i\colon L_i\to P^{k+n_i}$ such that $A_i\circ u_{\phi_i}=d\psi_i$
and  $A_j=A_i\times \id$ on $L_{ij}$. Then $\phi_i'=\widetilde{A_i}(\phi_i)$ will satisfy all requirements:

\[u_{\phi'_i}=A_i\circ u_{\phi_i}=d\psi_i\]
and, on $L_{ij}$, 
\[\phi'_j=\widetilde{A_j}(\phi_j)=(\widetilde{A_i}\times \id)(\phi_i\oplus q_{ij})=\phi'_i\oplus q_{ij}.\]

For the construction of $A_i$, we proceed as follows. We fix a metric on $E/V$, use the euclidean metric on $\R^{k+n_i}$ and consider
$p_i$ (resp. $q_i$) the orthogonal projection $E/V\times \R^k\times \R^{n_i} \to G$
on the image of $u_{\phi_i}$ (resp. $d\psi_i$). Consider the following sections of $P^{k+n_i}$:
\[B_i=\id + (0\times v\times 0)\circ p_i, \quad C_i=\id + (0\times 0\times u_{\phi_i}) \circ q_i, \quad A_i= C_i^{-1}\circ B_i.\]
We have 
\[B_i\circ u_{\phi_i}=\Pi\times v \times u_{\phi_i}=C_i\circ d\psi_i,\]
and hence
\[A_i\circ u_{\phi_i}=d\psi_i.\]
Finally, the property $A_j=A_i\times \id$ on $L_{ij}$ follows from construction since, on $L_{ij}$ we have
$u_{\phi_j}=u_{\phi_i}\times 0$.

\paragraph{Step 3: from infinitesimal to actual functions}

At this point we have a smooth $\QQ$-twisted lift $(\phi_i,q_{ij})$ of the Gauss map $G$ compatible
with the embedding $\psi:L\to M \times \R^k$. As we explained in the previous step, each $\phi_i$
can be viewed as a Legendrian plane field tangent to the isotropic embedding $L_i\to J^1(M\times \R^k)$.
We now upgrade these plane fields into germs of Legendrian submanifolds using the following result, whose proof is a straightforward differential topology exercise, using the Weinstein neighbourhood theorem: 
\begin{lem}\label{lem:weinsteinnhd}
Let $L$ be a compact isotropic submanifold in a contact manifold $V$ and $\phi$ a field of Legendrian planes along $L$
(i.e. Lagrangian planes in the contact structure with respect to its conformal symplectic structure) 
which is tangent to $L$ (namely for all $x\in L$, $T_x L \subset \phi(x)$). Then there exists a germ of a Legendrian submanifold $\Lambda$
along $L$ whose tangent spaces coincide with $\phi$ along $L$. Moreover if such a germ $\Lambda_A$ is already defined on a closed subset
$A$ of $L$, the germ $\Lambda$ may be assumed to extend $\Lambda_A$. \qed
\end{lem}

Applying Lemma~\ref{lem:weinsteinnhd}
repeatedly by induction on $i$, we obtain Legendrian submanifolds which correspond exactly to the $1$-jet graphs
of our desired generating functions $f_i$ above a small neighborhood $U_i$ of $\psi_i(L_i)$ in $M\times \R^{n_i}$
(by the transversality condition to the vertical subbundle).
We start with $\phi_1$ and obtain $(U_1,f_1)$. Then $f_2$ should coincide with $f_1\oplus q_{12}$ on a neighborhood
of $\psi_2(L_{12})$ in $M_{12}\times \R^{n_2}$. Using the last assertion in Lemma~\ref{lem:weinsteinnhd}
we can find such $f_2$ (up to negligible shrinking) defined on an open neighborhood $U_2$ of $\psi_2(L_2)$
with $U_2\cap (M_{12}\times \R^{n_2})\subset U_1\times \R^{n_{12}}$, and the process goes on.

This finishes the proof of Theorem~\ref{thm:twistedgirouxlatour} (and of Theorem~\ref{thmintro:twistedgirouxlatour}).

\section{The doubling trick and generating functions of tube type}\label{sec:doubling}

In this section, we prove Theorem \ref{thmintro:tgftube} which establishes the existence of (twisted) generating functions of tube type (see Definition~\ref{dfn:tgftube}) for a nearby Lagrangian $L$. The construction works in two steps:
\begin{enumerate}
\item construction of a twisted generating function linear at infinity for the $2$-copy of $L$,
\item separation of the $2$-copies to obtain a twisted generating function of tube type for $L$.
\end{enumerate}

A similar procedure was used by the third author in \cite{guillermou_quantization_2012} to associate a sheaf on $M\times \R$ to $L$.

\subsection{Twisted generating functions linear at infinity and the homotopy lifting property}

The functions of tube type which we shall produce are a special class of functions that are linear at infinity; the notion we find most convenient relies on distinguishing a coordinate direction (which we call $w$), and saying that a function is linear at infinity if it is the sum of this coordinate with a function which, outside of a compact set, depends only on the other variables:
\begin{dfn}\label{dfn:linear}
A function $f:M\times \R \times \R^n \to \R$ is \emph{linear at infinity} if
it is of the form:
\begin{equation}
 f(x,w,v)=w+g(x,v)+\epsilon(x,w,v)
\end{equation}
where $\supp(\epsilon) \to M$ is proper.
\end{dfn}

Observe that, once we distinguish the coordinate $w$,  the function $g$ is determined by $f$ via the formula
\[g(x,v)=\lim_{w\to \infty} f(x,w,v)-w,\]
and hence so is $\epsilon$.

The obvious left and right actions of direct sum with non-degenerate quadratic forms
$q:M \to \QQ_m$ do not preserve our notion of linearity at infinity because the support of $\epsilon$ fails to be proper when we multiply the domain by $\R^m$.
So we define a modified version $\cdot \oplus_b q$ of the usual action $\cdot \oplus q$, which depends on an auxiliary function $b$ on $M$, and which compensates for this failure. We will need appropriate cut-off functions.

Pick once and for all a smooth function $\chi:\R\to [0,1]$ satisfying
\begin{itemize}
\item $\chi$ is compactly supported,
\item $\chi=1$ near $0$,
\item $|\chi'(u)|<2|u|$ for all $u\neq 0$.
\end{itemize}
We then define
\[\chi_m:\R^m\to [0,1]\]
by the formula:
\[\chi_m(u_1,\dots,u_m)=\chi(u_1)\cdots\chi(u_m).\]

\begin{lem}
The functions $\chi_m$ satisfy the following properties:
\begin{itemize}
\item $\chi_{m_1+m_2}(u_1,u_2)=\chi_{m_1}(u_1)\chi_{m_2}(u_2)$ for all $(u_1,u_2)\in \R^{m_1}\times \R^{m_2}$,
\item $\chi_m=\chi_m \circ \sigma$ for any permutation of coordinates $\sigma\colon \R^m\to \R^m$,
\item $\|\nabla \chi_m (u)\|< 2\|u\|$ for all $u\in \R^m\setminus\{0\}$.
\end{itemize}
\end{lem}
\begin{dfn}\label{dfn:plusb}
Given a smooth function $b:M\to [1,\infty)$, a smooth map $q\colon M\to \QQ_m$ and a function $f:M\times \R \times \R^n \to \R$
linear at infinity, we define $f\oplus_b q \colon M\times \R \times \R^{n+m} \to \R$ by the formula:
\begin{equation}\label{eq:modified_action}
   (f\oplus_b q)(x,w,v,u) =w+g(x,v)+q(u)+\chi_m(b(x)^{-1}u)\epsilon(x,w,v) 
\end{equation}
where $g$ and $\epsilon$ are as in Definition~\ref{dfn:linear}.
We say that $b$ is a \emph{bound} for $f$ if
\begin{equation}
\max_{w,v}  |\epsilon(x,w,v)|\leq b(x) 
\end{equation}
for all $x \in M$.
\end{dfn}
When $M$ is compact, it is enough to consider functions $b$ which are constant. We list some basic properties of this construction:

\begin{lem}\label{lem:bound}
Let $f:M\times \R\times \R^n \to \R$ be a function linear at infinity and $b$ a positive function on $M$.
\begin{itemize}
\item $f\oplus_b q$ is linear at infinity,
\item if $b$ is a bound for $f$, then it is also a bound for $f\oplus_b q$,
\item if $b$ is a bound for $f$, then $f\oplus_b q$ and $f\oplus q$ have the same singular set, and coincide near this set,
\item $(f\oplus_b q_1) \oplus_b q_2= f\oplus_b(q_1\oplus q_2)$.
\end{itemize}
\end{lem}
\begin{proof}
For the first assertion, we have $f\oplus_b q=w+g'+\epsilon'$ with $g'=g+q$ and $\epsilon'=\epsilon\chi_m(b^{-1}\cdot)$
is compactly supported. The second follows from the inequality $|\chi_m|\leq 1$. For the third,
if $\frac{\del f \oplus_b q}{\del u}=0$ and $u\neq 0$, then 
\[\nabla q(u) + b^{-1}\epsilon \nabla \chi_m(b^{-1}u)=0\]
which is impossible since $\|\nabla \chi_m(u')\|<2\|u'\|=\|\nabla q(u')\|$ for $u'\neq 0$ (recall $q$ has eigenvalues $\pm 1$,
so $\|\nabla q(u)\|=2\|u\|$) and $b^{-2}|\epsilon|\leq b^{-1}|\epsilon|\leq 1$. So the singular
set of $f\oplus_b q$ is contained in $\{u=0\}$. But $f\oplus_b q= f\oplus q$ near $\{u=0\}$, hence
the result. The last assertion follows from the equality $\chi_{m_1+m_2}(u_1,u_2)=\chi_{m_1}(u_1)\chi_{m_2}(u_2)$.
\end{proof}

\begin{dfn}\label{dfn:tgenerating function linear at infinity}
A \emph{twisted generating function linear at infinity} over a manifold $M$
is given by data $(M_i, 1+n_i, f_i, q_{ij}, b)$, where
\begin{itemize}
\item $(M_i)_{i\in I}$ is a directed open cover,
\item $f_i\colon M_i\times \R\times \R^{n_i}\to \R$ is a generating function linear at infinity over $M_i$,
\item $b:M \to [1,\infty)$ is smooth and is a bound for each function $f_i$,
\item $q_{ij}$ is a map from $M_{ij}$ to $ \QQ_{n_j-n_i}$ such that $f_i \oplus_b q_{ij}=f_j$ on $M_{ij}$.
\end{itemize}
\end{dfn}

The above definition used the right action of quadratic forms on functions which are linear at infinity. We similarly define the left action but since the first factor of $\R$ is special we act on the right of this factor:
\begin{equation*}
  (q\oplus_b f)(x,w,u,v) = (f\oplus_b q)(x,w,v,u).
\end{equation*}

We now turn to the homotopy lifting property of twisted generating functions linear at infinity. The next result can be readily extracted from the work of Chekanov \cite{Chekanov1996} and Eliashberg-Gromov \cite{EliashbergGromov1998} (see also the proof by Chaperon and Théret \cite{chaperon}).

\begin{thm}\label{thm:chekanovEG}
Let $M$ be a manifold without boundary,
and $f^t:M \times \R^n \to \R$ a family of functions parametrised by $t\in[0,1]$ which generates
a family of Legendrian immersions $(\varphi\times z)^t\colon L\to J^1 M$.
For any compactly supported contact isotopy $(\theta^t)_{t\in [0,1]}$
of $J^1 M$ with $\theta^0=\id$, there exist compactly supported functions
$(\eta^t)_{t\in[0,1]}:M \times \R^{n} \times \R^{2m}\to \R$ with
$\eta^0=0$ such that the function $g^t\colon M \times \R^{n} \times \R^{2m}\to \R$, given by
\begin{align*}
&\quad g^t(x,u,x_1,y_1,\dots,x_m,y_m)\\
& = f^t(x,u)+x_1y_1+\dots+x_my_m+\eta^t(x,u,x_1,y_1,\dots,x_m,y_m),
\end{align*}
is a generating function for $\theta^t\circ(\varphi\times z)^t$. \qed
\end{thm}


The following lemma will be key to our proof of the homotopy lifting property for twisted generating functions linear at infinity.

\begin{lem}\label{lem:reorder}
Let $(b,M_i,f_i,q_{ij})$ be a well-ordered generating function linear at infinity over $M$ with indexing set $I$.
For any other total order $\leq'$ on $I$, there exists (up to a negligible shrinking of the cover)
a well-ordered twisted generating function linear at infinity $(b,M_i,f_i',q'_{ij})$ such that $(f'_i,q'_{ij})$ and $(f_i,q_{ij})$
generate the same Legendrian immersion and have equivalent underlying $\QQ$-twisted maps $M\to \Z$.
\end{lem}
\begin{proof}
Denoting the indexing set of the cover by $I=\{1,\dots,p\}$, it suffices to prove the existence of an equivalent generating function associated to the permutation of a pair of successive elements $(i, i+1)$. So for the new order we have $i+1<'i$. Let $m=2(n_{i+1}-n_i)$. The map $q_{i,i+1} \oplus -q_{i,i+1}\colon M_{i,i+1} \to \QQ_m$ is nullhomotopic, so (up to shrinking the cover) we can extend it to $M$; we denote the extension by $\delta q_{i,i+1}$. If $M_{i,i+1}$ is empty we take $\delta q_{i,i+1}$ to be a constant quadratic form on $\R^m$. We will add quadratic forms such that $n_j' = n_j$ when $j<i$ or $j=i+1$ and $n_j'=n_j+m$ when $j>i+1$ or $j=i$. We define $q'_{kl}$ as follows:
\begin{equation} q'_{kl}=
\begin{cases}
 q_{kl} &\text{ if   $k\geq' i$ or $l\leq' i+1$},\\
-q_{i,i+1}&\text{ if $k=i+1$ and  $l=i$},\\
-\delta q_{i,i+1} \oplus q_{kl} &\text{ if $k=i+1$ and $l>i+1$} ,\\
(q_{kl}\oplus \delta q_{i,i+1})\circ a_{k,i,l}^{-1}& \text{ if $k<i$ and $l\geq' i$},
\end{cases}
\end{equation}
where $a_{k,i,l}\colon \R^{n_{ki}}\times \R^{n_{il}}\times \R^m \to \R^{n_{ki}}\times \R^m \times \R^{n_{il}}$ is given
by $a_{k,i,l}(v,w,u)=(v,u,w)$. The permutation $a_{k,i,l}$ has the following effect: on $M_{k,i,l}$,
\[(q_{k,l}\oplus \delta q_{i,i+1}) \circ a_{k,i,l}^{-1}=(q_{k,i}\oplus q_{i,l}\oplus \delta q_{i,i+1})\circ a_{k,i,l}^{-1}=q_{k,i}\oplus \delta q_{i,i+1} \oplus q_{i,l}.\]

We can check that $q'_{jk} \oplus q_{kl}' = q_{jl}'$ when $j<'k<'l$. The reader can notice that
for consecutive indices $(k,k_+)$ with respect to the order $<'$, we have
\[q'_{k,k_+}=q_{k,k_+} \quad \text{for $k\not=i+1$,}  \qquad q'_{i+1,i}=-q_{i,i+1}. \]
This essentially imposes the above formula for $q'_{kl}$ with general indices $k,l$.

Similarly we define
\begin{equation} f'_k=
\begin{cases}
 f_k &\text{ if } k\leq' i+1,\\
 (f_k\oplus_b \delta q_{i,i+1})\circ a_k^{-1} &\text{otherwise,}
\end{cases}
\end{equation}
where $a_k\colon \R^{n_i}\times \R^{n_{ik}} \times \R^m \to \R^{n_i}\times \R^m \times \R^{n_{ik}} $
is given by $a_k(v,w,u)=(v,u,w)$. We can check that $f'_k \oplus_b q'_{kl} =f'_l$ for all $k<'l$.
Note, that this uses the properties of $\chi_m$ with respect to permutations and direct sums.
This generates the same Legendrian immersion due to Lemma~\ref{lem:bound}.

The equivalence of $(n_i,q_{ij})$ with $(n'_i,q'_{ij})$ can be proved directly by a similar construction
which we only sketch. Consider the open cover of $M\times [0,1]$ given by the open sets $M_i\times [0,2/3)$
and $M_i\times (1/3,1]$, indexed by $I\sqcup I$ ordered as
\[1<1'<2<2'<\dots<(i-1)'<i<i+1<(i+1)'<i'<i+2<\dots<p<p'.\]
We define the quadratic forms for consecutive indices in $I\sqcup I$ as:
\begin{align*}
q_{k,k'}&=0 \text{ for } k\neq i,\\
q_{k',k+1}&=q_{k,k+1} \text{ for } k\neq i,i+1,\\
q_{i',i+2}&=q_{i,i+2}
\end{align*}
and the same formula as before for $q_{kl}$ and $q_{k'l'}=q'_{kl}$. The remaining terms are now essentially
prescribed and can be made explicit using permutations as above.
\end{proof}

\begin{rem}
The above lemma works equally well for a $1$-parametric family of functions $f_i^t$, $t\in[0,1]$,
with $\QQ$-bundle $(M_i,q_{ij})$ independent of $t$.
\end{rem}

\begin{thm}\label{thm:twistedchekanov}
Let $L$ and $M$ be closed manifolds, $\varphi\times z\colon L\to J^1 M$ a Legendrian
immersion and $\theta^t:J^1 M\to J^1 M$, $t\in[0,1]$, a compactly supported contact isotopy with $\theta^0=\id$.
Given a twisted generating function linear at infinity $(b,f_i,q_{ij})$ for $\varphi\times z$,
there exists a twisted generating function linear at infinity $(c,g_k^t,p_{kl})$ for $\theta_t\circ(\varphi\times z)$
with equivalent underlying $\QQ$-twisted map from $M$ to $\Z$.
\end{thm}

\begin{proof}
Using Lemma~\ref{lem:well-order} we may assume that the cover $(M_i)_{i\in I}$
with respect to which the twisted generating function is defined is finite and
totally ordered, namely $I=\{1,\dots,p\}$ with the natural order.

As a first step, we fragment the isotopy $\theta^t$ and write it as a product
\[\theta^t=\theta_N^t\circ \cdots \circ \theta_k^t\circ \cdots \circ \theta_0^t\]
such that, for each $k$, $\theta_k^t$ has compact support in $J^1 M_i$ for some $i$.
The proof now works by induction on $k$. Assume we have constructed the desired twisted
generating function $(b,f_i^t,q_{ij})$ for $\theta_{k-1}^t\circ\cdots\circ\theta_0^t\circ (\varphi\times z)$
for some $k\geq 0$ (for $k=0$, we take the given twisted generating function independent of $t$).
Using Lemma~\ref{lem:reorder} we may assume that $\theta_k^t$ has compact support in $J^1 M_1$.

We now apply Theorem~\ref{thm:chekanovEG} to $f_1$ with $M=M_1$
to obtain $G^t_1:M_1\times \R \times \R^{2m}\times \R^{n_1}\to \R$, a generating function for
$\theta^t_k\circ\dots\circ\theta_0^t\circ(\varphi\times z)$ over $M_1$ of the form, $G^t_1=(h^m\oplus f^t_1) + \eta^t$, namely
\[G^t_1(x,w,x_i,y_i,v_1)=x_1y_1+\dots +x_my_m+f^t_1(x,w,v_1)+\eta^t(x,w,v,x_i,y_i,v_1),\]
with a compactly supported function $\eta$.
For a smooth function $c:M\to [1,+\infty)$, we set
\[G_1^{t,c}=(h^m\oplus_c f^t_1)+\eta^t,\]
which is linear at infinity.
Denote by $K$ the projection of the support of $\eta$ to $M_1$.
In view of Lemma~\ref{lem:bound}, if $c\geq b$ on $M_1$ and $c$ is large enough on $K$, the function $G_1^{t,c}$
still generates $\theta_k^t\circ\dots\circ\theta_0^t\circ(\varphi\times z)$ over $M_1$.
We pick such a function $c:M_1\to [1,+\infty)$ with the additional property that $c=b$ outside a compact neighborhood $K'$ of $K$ in $M_1$,
and define $g_1^t=G_1^{t,c}$.

For $i\geq 2$, we set $g_i^t=h^m \oplus_c f^t_i$ if $x\in M_i\setminus K'$
and $g^t_i=g^t_1 \oplus_c q_{1i}$ if $x\in M_{1i}$.
This agrees and glues correctly over $M_{1i}\setminus K'$ since on this set we have 
\begin{align*}
  h^m \oplus_c f^t_i = h^m \oplus_c f^t_1 \oplus_c q_{1i} = g^t_1 \oplus_c q_{1i}
\end{align*}
and $(c,g^t_i,q_{ij})$ is a twisted generating function linear at infinity
for $\theta^t_k\circ\dots\circ\theta_0^t\circ(\varphi\times z)$. Note that the underlying $\QQ$-twisted map
from $M$ to $\Z$ has changed only when we have applied Lemma~\ref{lem:reorder} and also
by the translation by $2m$ (due to the $h^m$ stabilization) but this does not change the equivalence class
in view of Lemma~\ref{lem:translation}.
\end{proof}

\subsection{The doubling trick}

\begin{dfn}\label{dfn:double}
For $s>0$, the \emph{$s$-double} of a Legendrian immersion $\varphi\times z\colon L\to J^1 M$
is the Legendrian immersion of the disjoint union $L\sqcup L$ given by $\varphi\times (z\pm s)$.
\end{dfn}
For a Legendrian embedding $L\to J^1 M$ and small $s>0$, this is also known as the $2$-copy
of $L$. The doubling trick converts a generating function for a Legendrian immersion into
a generating function linear at infinity for the $s$-double with $s>0$ small enough.

\begin{figure}[h]
  \centering
  \begin{tikzpicture}[
  declare function={
    func(\x)= (\x<=-2) * \x   +
    and(\x>-2, \x<=-1) * (-18 - 8*\x*\x*\x - 36*\x*\x - 48*\x)     +
    and(\x>-1,  \x<=1) * (\x*\x*\x - 3*\x) +
    and(\x>1,  \x<=2) * (18 - 8*\x*\x*\x + 36*\x*\x - 48*\x) +
                (\x>2) *  \x;
  }
]
\begin{axis}[
  axis x line=middle, axis y line=middle,
  ymin=-3, ymax=3, ytick={-3,...,3}, ylabel=$D(w)$,
  xmin=-3, xmax=3, xtick={-3,...,3}, xlabel=$w$,
]
\pgfplotsinvokeforeach{-2, -1, 1, 2}{
  \draw[dashed] ({rel axis cs: 0,0} -| {axis cs: #1, 0}) -- ({rel axis cs: 0,1} -| {axis cs: #1, 0});}
\addplot[blue, domain=-3:3, smooth]{func(x)};
\end{axis}
\end{tikzpicture} 
  \caption{The graph of the function $D$}
  \label{fig:function-D-graph}
\end{figure}
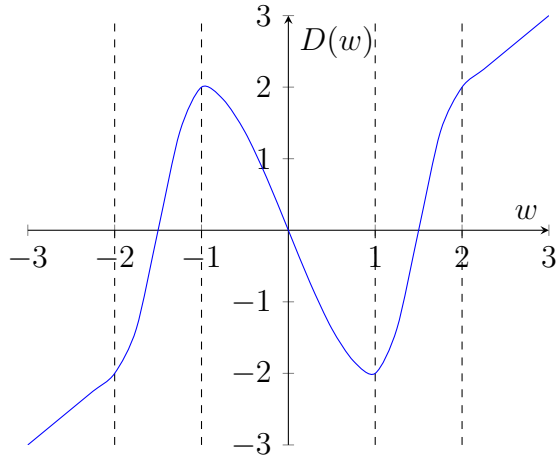

To implement this trick (and also to discuss tubes later) we pick a
smooth function $D: \R \to \R$ such that:
\begin{itemize}
\item $D(w)= w^3-3w$ for $|w|\leq 1$,
\item $D'(w)>0$ for $|w|>1$, and
\item $D(w) = w$ when $|w|>2$.
\end{itemize}
The function
\begin{align} \label{d:eq}
  D_t(w) = w +(\tfrac14 +t)(D(w)-w)
\end{align}
has no critical points for $-\tfrac14 \leq t<0$ and two critical points $\pm r(t)$ close to 0 (explicitly $r(t)=\frac{4}{\sqrt{3}}\sqrt{\frac{t}{1+4t}}$)
with critical values $\pm s(t)$ close to 0 for small $t>0$.

Now let $U \subset M \times \R^n$ be open and $g:U \to \R$ a generating function for a Legendrian $L \subset J^1M$ with $L \to M$ proper. Let $g': M \times \R^n \to \R$ be an arbitrary smooth extension of the germ of $g$ around $\Sigma_g$. This may generate a lot more than $L$, but $\Sigma_g \subset \Sigma_{g'}$ is isolated. Let $\alpha : M\times \R^n \to [0,1]$ be a smooth function which equals $1$ in a neighborhood of $\Sigma_g$, and whose support is proper over $M$ and is contained in the locus where $g$ and $g'$ agree. By construction the support of $\alpha$ intersects $\Sigma_{g'}$ along $\Sigma_g$. For small $t>0$ we define a function
\begin{align}
  f^t & \colon M \times \R\times \R^n  \to \R \\   \label{eq:3}
    f^t(x,w,v) & = g'(x,v) + w + (\tfrac14+t)\alpha(x,v)(D(w)-w) .   
\end{align}
These functions are linear at infinity since the support of $\alpha(x,v)(D(w)-w)$ maps properly to $M$.

\begin{lem} \label{compact-generation}
For any compact $K \subset M$ and $t>0$ small enough, the function $f^t$ generates the $s(t)$-double of $L$ over $K$.
\end{lem}

\begin{proof}
  We first check that for small $t>0$ the fiberwise critical set over $K$ of $f^t$ is precisely
  \[\Sigma_{f^t}=\left\{(x,w,v) \in M\times \R\times \R^k; \; (x,v)\in \Sigma_g, \, D_t'(w) = 0 \right\}\]
  which is two copies of $L$. In the following all functions are restricted to the sets over $K$ (hence proper to $M$ implies compact).

  Note that on the closure $W$ of the set $0 < \alpha < 1 $ the gradient $\frac{\del g'}{\del v}=\frac{\del g}{\del v}$ is non-zero and hence there is a $c>0$ such that
  \begin{equation} \label{eq:bound}
    \norm{\frac{\del g'}{\del v}} > c
  \end{equation}
  on this set. Having $\frac{\del f^t}{\del w}(x,w,v)=0$ implies $\alpha(x,v)\geq \tfrac{1}{1+4t}>0$, so $(x,v)$ lies in the support of $\alpha$.
Setting $t'$ such that $\tfrac14+t'=(\tfrac14+t)\alpha$ we have $0<t'\leq t$ and $w=\pm r(t')$ tends to $0$ as $t$ goes to $0$.
If furthermore $\frac{\del f^t}{\del v}(x,w,v)=0$, then
  \begin{equation}\label{ineq}
    \left\|\frac{\del g'}{\del v}(x,v)\right\|= |D(w)-w|\cdot (\tfrac14+t) \cdot \left\|\frac{\del \alpha}{\del v}(x,v)\right\|
  \end{equation}
  But $\frac{\del \alpha}{\del v}$ is bounded on $W$ and $|D(w)-w|\to 0$ as $w\to 0$, so in view of \eqref{eq:bound}, for $t$ small enough, 
  the latter equality can hold only on the region where $\alpha=1$. This means that $f^t(x,w,v)=g(x,v) + D_t(w)$. So the singular set of $f^t$ is as asserted, and
  we compute that for $(x,w,v)\in \Sigma_{f^t}$,
  \[f^t(x,w,v)=f(x,v)\pm s(t).\]
It is easily checked that $\Sigma_{f^t}$ is transversely cut out since $g$ transversely cuts out $\Sigma_g$.
\end{proof}

We note that the following corollary only depends on Theorem~\ref{thm:girouxlatour} and the short construction in this section.

\begin{cor} \label{stable-trivial-linear}
  Let $M$ be a manifold, $L$ a closed manifold and $L\to J^1 M$ a Legendrian immersion
with nullhomotopic stable Gauss map. Then for small $s>0$, the $s$-double of $L$ admits a generating function linear at infinity.
\end{cor}

\begin{proof}
  This follows from Theorem~\ref{thm:girouxlatour} and the lemma above using an $\alpha$ picked to have compact support over $M$.
\end{proof}

However, the more interesting result is the following.

\begin{thm}\label{thm:twisted-double}
  Let $L$ be a closed manifold and $\varphi\times z\colon L\to J^1 M$ a Legendrian immersion which admits
  a generating function twisted by $h:M\to \Lambda_0(\C^\infty)$. Then for small $s>0$, the $s$-double of $L$
  admits a generating function linear at infinity twisted by $h$.
\end{thm}

\begin{proof}
  By Lemma~\ref{lem:well-order} we may start with a twisted generating function $(M_i,n_i,g_i,q_{ij})_{i\in I}$
  which is well-ordered, namely $I = \{1,\dots,k\}$ with the natural order and $n_i \leq n_{i+1}$. We set $n= n_k = \max_{i} n_i$. We write a point
  of $M \times \R^n$ as $(x,v_1, v_{12}, v_{23},\ldots, v_{k-1,k})$ and use the notation
$v_i = (v_1, v_{12}, \ldots , v_{i-1,i})$, $v_{ij} = (v_{i,i+1},\ldots, v_{j-1,j})$.
After slightly shrinking the cover $(M_i)$, we can assume that $g_i$ is defined on a neighborhood
of the compact set
\[\overline{\Sigma_i}=\Sigma_{g_i}\cap \overline{M_i}\times \R^{n_i},\]
in $M\times \R^{n_i}$ and $q_{ij}$ is defined on a neighborhood of $\overline{M_{ij}}$ in $M$,
so that $g_i\oplus q_{ij}=g_j$ on a neighborhood of $\overline{\Sigma_j}\cap (\overline{M_{ij}}\times \R^{n_j})$
in $M\times \R^{n_j}$.

The first step is to construct smooth functions $g'_i$ defined near $\overline{M_i}\times \R^{n_i} \to \R$
which coincide with $g_i$ near $\overline{\Sigma_i}$ and satisfy the relation $g'_i \oplus q_{ij}=g'_j$
near $\overline{M_{ij}}\times \R^{n_j}$ for all $i<j$. We proceed by induction on $i$.
We pick $g'_1$ to coincide with $g_1$ near $\overline{\Sigma_1}$. Since $g_2$ agrees with $g_1\oplus q_{12}$ near $\overline{\Sigma_2}\cap \overline{M_{12}}\times \R^{n_2}$, we can find $g'_2$ defined on a neighborhood of $\overline{M_2}\times \R^{n_2}$ which coincides with
$g'_1 \oplus q_{12}$ on a neighborhood of $\overline{M_{12}}\times \R^{n_2}$ and with $g_2$ on a neighborhood of $\overline{\Sigma}_2$.
We proceed by induction until we have defined all $g'_i$. Note that $(g'_i,q_{ij})$ possibly has
a larger singular set $\Sigma'_i$ but $\Sigma_i$ is isolated in $\Sigma'_i$.

Let us pick a constant $b\geq 1$ such that $b\geq \max_{w\in \R} |D(w)-w|$. Using $\Op(\_)$ to denote an arbitrary open neighbourhood, we now similarly inductively pick compactly supported smooth functions $\alpha_i\colon \Op(\overline{M}_i) \times \R^{n_i} \to [0,1]$
which equal $1$ near $\Sigma_i$ and $0$ on $\Sigma'_i\setminus \Sigma_i$ and such that for $x$ near $\overline{M}_{ij}$
we have
\begin{align*}
  \alpha_{i}(x,v_i) \chi_{q_{ij}}(b(x)^{-1} v_{ij}) = \alpha_j(x,v_j).
\end{align*}
This is possible since $\Sigma'_j$ coincides with $\Sigma'_i\times \{0\}$ above $\Op(\overline{M}_{ij})$.

Now we define $f_i(x,w,v_i)=w+g'_i(x,v_i)$ and $f_i^t$ as in Equation~\eqref{eq:3}, namely
\[f_i^t(x,w,v_i)=w+g'_i(x,v_i)+\alpha_i(x,v_i)(\tfrac{1}{4}+t)(D(w)-w).\]

Lemma~\ref{compact-generation} above (applied with $M=\Op(\overline{M}_i)$ and $K=\overline{M_i}$)
now ensures that $f_i^t$ generates the $s(t)$-double
of $\varphi\times z$ above $\overline{M}_i$ for small $t>0$ (the same for each $i$).

Finally we check the relation $f_i^t \oplus_b q_{ij} = f_j^t$ on $\overline{M_{ij}}$:

\begin{align*}
f_i^t\oplus_b q_{ij}&=w+g'_i+q_{ij}+\chi_{q_{ij}}(b^{-1}\cdot)\alpha_i (\tfrac14+t)(D(w)-w)\\
&=w+g'_j+\alpha_j(\tfrac14+t)(D(w)-w).
\end{align*}

Remark that all of the above works with any value of $b\geq 1$. But in Definition~\ref{dfn:tgenerating function linear at infinity}
we have required that $b$ is a bound for $f_i$, which is the case with
our choice of $b$ (and $t<\tfrac34$).
\end{proof}

\subsection{Difference function}

\begin{dfn}\label{dfn:diff}
Given a function $f\colon M\times \R^n \to \R$, the \emph{difference function}
$\delta f \colon M \times \R^n\times \R^n\to \R$ is defined by
\[\delta f(x,v_1,v'_1,\dots,v_n,v'_n)=f(x,v_1,\dots,v_n)-f(x,v'_1,\dots,v'_n).\]
\end{dfn}
Note that we alternate the variables as it will be more convenient later. 
This is a classical tool in the theory of generating functions. 
We will use it to derive the homological condition in Proposition~\ref{prop:tuberecognition}
in the case of Lagrangian embeddings. Observe that $\delta f$ is not strictly speaking a generating function
since it does not satisfy the transversality condition of Definition~\ref{dfn:gf} in general, hence we will
treat difference functions merely as functions.

If $f$ generates a Legendrian embedding $L\to J^1 M$, then the global critical points of $\delta f$
correpond to Reeb chords of $L$ and a Morse-Bott copy of $L$ in the level set $\{\delta f=0\}$.
In particular, when $L$ is a lift of an \emph{embedded} Lagrangian submanifold, then the global
critical points of $\delta f$ consists merely of this copy of $L$.

A first issue is that $\delta f$ is not necessarily linear at infinity when $f$ is. In particular in the twisted case,
the $\oplus_b$ operation cannot be used directly with $\delta f_i$. We are led to introduce the following
definition.

\begin{dfn}\label{dfn:deltalinear}
A function $F:M\times \R^2 \times (\R^2)^n \to \R$ is called \emph{$\delta$-linear at infinity} if
it is of the form:
\begin{align*}
& \quad F(x,w,w',v_1,v'_1,\dots,v_n,v_n')\\
& =w-w'+G(x,v_1,v_1',\dots,v_n,v_n')+\epsilon(x,w,v)-\epsilon'(x,w',v')
\end{align*}
where $\supp(\epsilon) \to M$ and $\supp(\epsilon') \to M$ are proper.
\end{dfn}

Observe that the functions $G, \epsilon$ and $\epsilon'$ are determined by $F$ since the functions
\[F(x,w,0,v_1,0,\dots,v_n,0) \quad \text{ and } \quad F(x,0,w',0,v'_1,\dots,0,v'_n)\]
are linear at infinity and thus determine $\epsilon$ and $\epsilon'$.

We now slightly modify the operation $\oplus_b$ to adapt it to functions $\delta$-linear at infinity.

\begin{dfn}\label{dfn:plusbdelta}
Given a smooth function $b:M\to [1,\infty)$, a smooth map $Q\colon M\to \QQ_{2m}$ and a function $F:M\times \R^2 \times (\R^2)^n \to \R$
which is $\delta$-linear at infinity, we define $F\oplus^\delta_b Q \colon M\times \R^2 \times (\R^2)^n \times (\R^2)^m \to \R$ by the formula:
\begin{align*}
& \quad (F\oplus^\delta_b Q)(x,w,w',v_1,v'_1,\dots,v_n,v'_n,u_1,u'_1,\dots,u_n,u'_n)\\
& =w-w'+G(x,v_1,v'_1,\dots,v_n,v_n')+Q(x,u_1,u'_1,\dots,u_n,u'_n)\\
&+\chi_{2m}(b(x)^{-1}(u_1,0,\dots,u_m,0))\epsilon(x,w,v) \\
&+\chi_{2m}(b(x)^{-1}(0,u'_1,\dots,0,u'_m))\epsilon'(x,w',v')
\end{align*}
where $G,\epsilon$ and $\epsilon'$ are as in Definition~\ref{dfn:deltalinear}.
We say that $b$ is a \emph{bound} for $F$ is $|\epsilon|\leq b$ and $|\epsilon'| \leq b$.
\end{dfn}

\begin{dfn}\label{dfn:deltaq}
For $q\in \QQ_n$, we define $\delta q \in \QQ_{2n}$ by the formula
\[q(u_1,u'_1,\dots,u_n,u'_n)=q(u_1,\dots,u_n)-q(u'_1,\dots,u'_n).\]
\end{dfn}

We note the following important lemma which says that, even though $\QQ$ is not a group,
$-q$ serves as a homotopy inverse of $q$.

\begin{lem}\label{lem:deltaq}
The map $\delta \colon \QQ \to \QQ$ is a monoid map and it is homotopic to the map $q\mapsto h^{\dim q}$ (where $h=2xy \in \QQ_2$)
through monoid maps.
\end{lem}
\begin{proof}
A direct check shows that $\delta(q_1\oplus q_2)=\delta q_1 \oplus \delta q_2$ for all $q_1,q_2\in \QQ$.

For $t\in[0,\pi/2]$ and $q\in \QQ$ we set 
\[(\delta q)^t = \cos(t)\delta q + \sin(t) h^{\dim q}.\]
The matrix of $(\delta q)^t$ has block form
\[\begin{pmatrix} \cos(t)q & \sin(t) \\ \sin(t) & -\cos(t)q \end{pmatrix}\]
in an appropriate basis, so $\det((\delta q)^t)=\det(-\cos^2(t) q^2 - \sin^2(t))\neq 0$ since $q^2$ is positive definite.
Moreover $(\delta q)^t$ has eigenvalues $\pm 1$, hence $\delta q^t \in \QQ$ for all $t\in[0,\pi/2]$.
Finally, we check that
\[(\delta(q_1\oplus q_2))^t=(\delta q_1 \oplus \delta q_2)^t=(\delta q_1)^t\oplus (\delta q_2)^t.\]
\end{proof}

The proof of the next result is similar to that of Lemma \ref{lem:bound}, we omit it.
\begin{lem}\label{lem:plusbdeltaproperties}
The  operation $\oplus^\delta_b$ satisfies the following properties:
\begin{itemize}
\item $(F \oplus^\delta_b Q_1) \oplus^\delta_b Q_2= F \oplus^\delta_b(Q_1\oplus Q_2)$,
\item if $f$ is linear at infinity, then $\delta f$ is $\delta$-linear at infinity,
\item $\delta f \oplus^\delta_b \delta q=\delta(f\oplus_b q)$,
\item if $b$ is a bound for $F$, then $F\oplus^\delta_b Q$ has the same singular set
as $F\oplus Q$ and coincides with $F\oplus Q$ near this set.
\end{itemize} \qed
\end{lem}

In the same way that the operation $\oplus_b $ made it possible to formulate Definition \ref{dfn:tgenerating function linear at infinity}, we now formulate the $\delta$-linear analogue:
\begin{dfn}\label{dfn:tfunction deltalinear at infinity}
A \emph{twisted function $\delta$-linear at infinity} over a manifold $M$
is given by data $(M_i, 2+2n_i, F_i, Q_{ij}, b)$, where
\begin{enumerate}
\item $(M_i)_{i\in I}$ is a directed open cover,
\item $F_i\colon M_i\times \R^2\times (\R^2)^{n_i}\to \R$ is a function $\delta$-linear at infinity over $M_i$,
\item $b:M \to [1,\infty)$ is smooth and is a bound for each function $F_i$,
\item $Q_{ij}$ is a map from $M_{ij}$ to $ \QQ_{2(n_j-n_i)}$ such that $F_i \oplus^\delta_b Q_{ij}=F_j$ over $M_{ij}$.
\end{enumerate}
\end{dfn}

For the statement of the next result, we observe that Lemmas~\ref{lem:deltaq} and \ref{lem:plusbdeltaproperties} imply that,
for any twisted generating function $(b,f_i,q_{ij})$ linear at infinity, the difference $(b,\delta f_i, \delta q_{ij})$
is a twisted function $\delta$-linear at infinity. Using Lemma~\ref{lem:trivialQbundle}, we can obtain an associated untwisted function:

\begin{lem}\label{lem:difference}
 Up to refining the cover with respect to which $(b,f_i,q_{ij})$ is defined, there exist quadratic forms $Q_i\colon M_i\to \QQ$ such that the functions
  $\delta f_i \oplus_b^\delta Q_i$ glue into a genuine function $\delta$-linear at infinity.
\end{lem}

\begin{proof}
We simply apply Lemma~\ref{lem:trivialQbundle} to the nullhomotopy of $\delta q_{ij}$ provided by Lemma~\ref{lem:deltaq}.
\end{proof}

Finally, recall that we can associate to each generating function $f$ which is linear at infinity the relative homology groups $ H_*(f\leq a_+,f \leq a_-)$ of the sublevel set associated to each interval $[a_-,a_+]$. The linearity condition implies that these groups do not depend on $-a_-$ and $a_+$ whenever they are sufficiently large, and we write $a_- = -\infty$ and $a_+ = +\infty$ in this case. The following result provides a computation of the homology groups associated to difference functions, and is formulated in terms of the derived category  $D^b(\Z)$ of bounded complexes of abelian groups; homology groups are viewed as objects of this category corresponding to complexes with trivial differential.
\begin{lem}\label{lem:homologydifference}
  Let $f:\R^n\times \R \to \R$ be a function linear at infinity.  If there
  exists a constant $c>0$ such that the critical values of $f$ are contained in the union of intervals
  $(-2c,-c)\cup (c,2c)$, then there is a chain complex $C$ to that
  \begin{equation}
    H_*(\delta f\leq -2c, \delta f\leq -\infty)\simeq  H_*(f\leq 0,f \leq -\infty) \overset{L}\otimes C.
  \end{equation}
\end{lem}
\begin{proof}
 This can be seen using Morse theory. We assume $f$ is Morse and pick  a gradient vector field $X$ for $f$. The linearity condition at infinity means we can use a pseudo-gradient which equals $\del_w$ outside a compact set to define the Morse complex of $f$, which can be used to compute the relative homologies of the sub-level sets.

  The Morse complex $C(f,X)$ admits a decomposition
  $C_-\oplus C_+$ whith differential
  $d=\left(\begin{smallmatrix} d_+ & 0 \\ d_{\pm} &
      d_- \end{smallmatrix}\right)$, where $C_+$ (resp. $C_-$) is generated by
  critical points of positive (resp. negative) critical value.  Here $(C_-,d_-)$
  computes $H_*(f \leq 0,f \leq -\infty)$. Now, due to the bounds on the critical
  values of $f$, the Morse subcomplex $C_{\leq -2c}(\delta f, X\oplus (-X))$
  generated by critical points of $\delta f$ with critical values less than or equal
  to $-2c$ is isomorphic to
  $(C_-(f)\otimes C_+(f)^*,d_-\otimes \id + \id \otimes (d_+)^*)$.  Since the
  Morse complex is free over $\Z$, this gives the result with
  $C=C_+(f)^*$, where we have used the fact that complexes
  of abelian groups are isomorphic in $D^b(\Z)$ to their cohomology to identify $C_-(f)$ with its homology.
\end{proof}

\subsection{Proof of Theorem~\ref{thmintro:tgftube}}

Recall that the function $D: \R \to  \R$ from the previous section has two critical points with critical values $\pm 1$ and Morse index 0 and 1. Since $|D(w)-w| \leq 4$, for any quadratic form $Q\in \QQ_n$ the function $f=D\oplus_b Q$, whenever $b \geq 4$, also has only two critical points with Morse index $\ind(Q)$ and $\ind(Q)+1$ where $\ind(Q)$ is the negative index of $Q$. In particular $0$ is a regular value of $f$ and the sublevel set $\{f\leq 0\}$ is obtained from $\{f\leq c\}$, for any $c<-2$, by attaching a trivial
handle of index $\ind(Q)$ (here by a trivial handle we mean precisely such a handle that is in cancellation position
with a handle of index one more, in particular the attaching sphere of the handle
bounds a disk and has trivial framing). We refer to such a sublevel set $\{D\oplus_b Q \leq 0\}$ as a \emph{rigid tube}.
In the next definition, we consider more general functions which are linear at infinity, with possibly
more critical points, but with a sub $0$-level set which is isotopic to a rigid tube.
\begin{dfn}
A function $f \colon \R \times \R^n \to \R$ is \emph{of tube type} if there is a function
$F\colon[0,1]\times \R \times \R^n\to \R$ which is linear at infinity, such that $0$ is a regular value of the function $(w,v)\mapsto F(t,w,v)$ for all $t\in [0,1]$, and the boundary values for $t = 0$ and $t=1$ are respectively given by $f$ and  $D \oplus_4 Q$ for some quadratic form $Q$.
\end{dfn}
Note that we are not really discarding the function on the set $\{f>0\}$. Indeed, to define the action of $\QQ$ on such functions we need this part.

We can then formulate the version of the notion of function of tube type associated to a manifold $M$:
\begin{dfn}\label{dfn:tgftube}
 A generating function $f\colon M\times \R \times \R^n \to \R$
 is of \emph{tube type over $M$} if it is linear at infinity and the function $f_x \colon \R \times \R^n \to \R$ is of tube type for each $x\in M$.
 
 We say that a generating function $f$ of tube type \emph{tube generates} a Legendrian $L \to J^1(M)$ if the restriction of $f$ to $\{f\leq 0\}$ generates $L$.
As in Definition~\ref{dfn:gf} we also say that $f$ tube generates the exact Lagrangian immersion $L\to J^1 M \to T^* M$.

 A \emph{twisted generating function of tube type} is a twisted generating function $(b,M_i,f_i,q_{ij})$ which is 
linear at infinity, such that each $f_i$ is a generating function of tube type
over $M_i$.
\end{dfn}

We start with a proposition allowing to recognize a generating function of tube type
from a homological condition. This relies on Smale's h-cobordism theorem. For the statement, we recall that $h=2xy$.

\begin{prop}\label{prop:tuberecognition}
  If $f:\R \times \R^n \to \R$ is a function which is linear at infinity such that $0$
  is a regular value and $H_*(f\leq 0,f\leq -\infty)\simeq \Z[-i]$ for some integer $i$,
then there exists an integer $N$ such that $h^N \oplus_b f$ is of tube type (for large enough $b$).
\end{prop}

\begin{proof}
  For large $b>1$, we can choose a perturbation of $f$ (and hence of $q\oplus_b f$)
  making it Morse so that there is a gradient-like vector field with the
  property that the flow lines between all critical points lie inside the set
  where $q\oplus_b f=q\oplus f$. Hence the level set homology is suspended by the
  index of $q$. If we prove that $q\oplus_b f$ is of tube type, then $\delta q \oplus_b f$
  will also be of tube type and $\delta q$ is homotopic to $h^N$ for some $N$ according to Lemma~\ref{lem:deltaq}.
  Hence we can stabilize $f$ arbitrarily and we can assume:
  \begin{itemize}
  \item $2 \leq \ind(q) \leq n-2$ for all critical point $q$ of $f$ in $\{f\leq 0\}$,
  \item $4\leq 2i\leq n-3$.
  \end{itemize}
  As $D\oplus_4 Q$ also satisfy these when $Q$ has Morse index $i$ we finish the proof by arguing that there is a deformation through functions linear at infinity between any such two functions, say $f_0$ and $f_1$, keeping $0$ regular.

  If $f_j(x,w,v)=w + g_j(x,v) + \epsilon_j(x,w,v)$ as in Definition~\ref{dfn:linear} then we may first deform both functions by
  \begin{align*}
    f_{j,t}(x,w,v)=w + (1-t\varphi(x,v))g_j(x,v) + \epsilon_j(x,w,v)
  \end{align*}
  where $1-\varphi$ has compact support and the support of each $\epsilon_j,j=0,1$ is contained in the set where $\varphi=0$. As all critical points must have $\epsilon_j\neq 0$ it follows that these stay constant for both functions. Using such deformations we may thus assume that $f_j^{-1}(0)$ agrees with $\{w=0\}$ outside a compact set. 

  It is then enough to find a compactly supported isotopy taking $f^{-1}_0(0)$ to $f^{-1}_1(0)$. Indeed after this isotopy
the convex interpolation between $f_0$ and $f_1$ will remain transverse to $0$. This isotopy will be constructed
via the h-cobordism theorem.

We pick $c>0$ large enough so that $\{w \leq -c\} \subset \{f_j \leq 0\}$ for each $j=1,2$. Note that the relative homotopy type $\{f_j\leq 0\}/\{w\leq -c\}$ is a based sphere of dimension $i$. Indeed, as there are no 1-handles in our Morse functions the homology supported in one degree implies this by Hurewicz's theorem.

We then pick two smooth disc embeddings $(D^i,\partial D^i) \to (\{f_j < 0\},\{w\leq -c\})$ which are also relative homotopy equivalences. As the space of disc embeddings $(D^i,\partial D^i) \to (\R^{n+1},\{w\leq -c\})$ is connected by the assumption $2i\leq n-3$ we can apply an ambient isotopy to one of the functions and assume that the two maps from $D^i$ into $\R^{n+1}$ are in fact equal.

Let $T$ be a smoothing of a small tubular neighborhood of this disc union the set $\{w\leq -c\}$. We may assume that these sits inside both $\{f_j < 0\}$. Consider the complements $W_j = \{f_j\leq 0\} \setminus \Int T$. The relative homology of each $(\{f_j\leq 0\},T)$ now vanishes implying that $W_j$ is a homology-cobordism. As $W_j$ is obtained from the simply connected $\{f_j\leq 0\}$ by carving out a tubular neighborhood around a $D^i$ with $2i<n-3$ it is simply connected. Its two boundaries are also simply connected using the above bounds on Morse indices, so $W_j$ is in fact an $h$-cobordism. It follows that both $W_j$ are compactly supported $h$-cobordisms from $\partial T$ to $\{f_j=0\}$, and the claim follows from the $h$-cobordism theorem.
\end{proof}

Our goal for the rest of this subsection is to prove the following theorem, which is a precise
version of Theorem~\ref{thmintro:tgftube}.

\begin{thm}\label{thm:existencetgftube}
  Let $M$ and $L$ be closed manifolds, and $\varphi \times z \colon L \to J^1 M$ a Legendrian lift of a Lagrangian embedding
  $\varphi:L\to T^* M$. If $s_1>\max_L|z|$, then the $s_1$-double admits a twisted generating function of tube type.
  In particular this twisted generating function tube generates $\varphi \times (z-s_1)$.
\end{thm}

In order to prove Theorem \ref{thm:existencetgftube}, we need two preliminary results of homological nature. Given a ring $A$, we write $D^b(A_M) $ for the derived category of (complexes of) sheaves of $A$-modules on $M$. The first result is a homological characterisation of those complexes which are equivalent to the constant rank-$1$ sheaf:
\begin{lem}\label{lem:rankone}
 If $M$ is a compact connected orientable manifold of
  dimension $n$ and $G\in D^b(A_M)$ such that
  $\R\Gamma(M; G) \simeq \R\Gamma(M;A_M)$ in the derived category of $A$-modules and $H^i G$ are constant.  Then
  $G \simeq A_M$ in $D^b(A_M)$.
\end{lem}
\begin{proof}
Recall from \cite[Remark 1.7.6]{KS_1994} that we have truncation endofunctors $\tau_{\geq a}$ and $\tau_{\leq b}$ of $D^b(A_M)$
together with morphisms $F\to \tau_{\geq a} F$ and $\tau_{\leq b} F \to F$ inducing
isomorphisms in cohomological degrees $\geq a$ and $\leq b$ respectively. They come with a distinguished
triangle $\tau_{\leq a} F \to F \to \tau_{\geq a+1} F\overset{+1}{\to}$.

Let $a$ be the minimal integer such that $H^a G\neq 0$. We have a distinguished triangle
$G'\to G \to G''$ with $G''=\tau_{\geq a+1} G$ and $G'=H^a G[-a]$.
Since $H^a G$ is constant and non zero, we have $H^a(M;G')=H^0(M;H^a G)\neq 0$.
Since $G''$ is concentrated in degree greater than or equal to $a+1$, we have $H^i(M,G'')=0$
for $i< a+1$. The long exact sequence
\begin{equation}\dots \to H^i(M;G')\to H^i(M;G) \to H^i(M;G'') \to \dots \end{equation}
now implies that $H^a(M;G)\neq 0$. But by assumption $H^a(M;G)\simeq H^a(M;A)$,
hence $a\geq 0$.

Similarly, take $b$ the maximal integer such that $H^b G \neq 0$, and consider
the triangle $G'\to G \to G''$ with $G'=\tau_{\leq b-1} G$ and $G''=H^b G[-b]$.

Then $H^{n+b}(M;G'')=H^n(M;H^b G)$ is non-zero since $H^b G$ is constant and $M$ is orientable.
Moreover $H^i(M;G')=0$ if $i>b-1+n$ since $G'$ is concentrated in degrees $\leq b-1$.
The long exact sequence gives $H^{b+n}(M;G)\neq 0$, and hence $b\leq 0$.
Therefore $H^i G=0$ for $i\neq 0$ and $H^0 G$ is constant by hypothesis, hence $G\simeq E_M$
for some $A$-module $E$ (see \cite[Proposition 1.7.2]{KS_1994}). Hence $A=H^0(M;A_M)=H^0(M;G)=H^0(M;E_M)\simeq E$ since $M$ is connected,
and $G\simeq A_M$.
\end{proof}

The next result is a characterisation of free rank-$1$ abelian groups in terms of their properties with respect to the tensor product of complexes of modules:
\begin{lem}\label{lem:tensor}
  If $A$ and $B$ are bounded complexes of abelian groups with finite rank cohomology, which satisfy
  $(A\overset{L}{\otimes} B)\overset{L}{\otimes} (\Z/p\Z) \simeq \Z/p\Z$ for
  all primes $p$, then $A\simeq \Z[d]$ for some $d\in \Z$.
\end{lem}
\begin{proof}
  We can decompose $A$ and $B$ as a sum of torsion free and torsion parts, say
  $A \simeq A_f \oplus A_t$, $B \simeq B_f \oplus B_t$ (we first decompose
  $A = \bigoplus_i H^iA[-i]$ and then decompose the $H^iA$'s).  We first choose
  a prime $p$ which annihilates the torsion $A_t$ and $B_t$.  Then
  $(A\overset{L}{\otimes} B)\overset{L}{\otimes} (\Z/p\Z) \simeq A_f \otimes B_f
  \otimes (\Z/p\Z)$  and we deduce $A_f \simeq \Z[d]$, $B_f \simeq \Z[-d]$, for
  some $d\in \Z$.  Now we assume that $B_t\not=0$ and choose a non trivial
  direct summand $(\Z/n\Z)[i]$. Let $q$ be a prime dividing $n$. Then
  $(\Z/n\Z) \overset{L}{\otimes} (\Z/q\Z) \simeq (\Z/q\Z \xrightarrow{n = 0} \Z/q\Z)$ has
  non zero cohomology in two degrees. But
  $(A_f \overset{L}{\otimes}\Z/n\Z[i] ) \overset{L}{\otimes} (\Z/q\Z)$ is a
  direct summand of $(A\overset{L}{\otimes} B)\overset{L}{\otimes} (\Z/p\Z)$ and
  we have a contradiction.  Hence $B_t=0$ and, in the same way, $A_t=0$.
\end{proof}

\begin{proof}[Proof of Theorem \ref{thm:existencetgftube}]
According to \cite{Kragh2013} and \cite{Abouzaid2012a} (or \cite{guillermou_quantization_2012})
we know that $L\to M$ is a homotopy equivalence. In particular the stable Gauss map
of $L$ factors up to homotopy through a map $h:M\to\Lambda_0(\C^\infty)$.
Theorem~\ref{thm:twistedgirouxlatour} provides a twisted generating function $(g_i,q_{ij})$ for $L$
twisted by $h$.

Theorem~\ref{thm:twisted-double} transforms $(g_i,q_{ij})$ into a twisted generating
function linear at infinity $(b,f_i,q_{ij})$ for the $s_0$-double of $\varphi\times z$
for some small $s_0>0$. Let $s_1>\max_L|z|$.
We pick a compactly supported contact isotopy
$(\theta_s)_{s\in [s_0,3s_1]}$ of $J^1 M$ such that $\theta_{s_0}=\id$ and
\[\theta_s \circ (\varphi\times (z\pm s_0))=\varphi\times (z \pm s).\]
We apply Theorem~\ref{thm:twistedchekanov}
to $(b,f_i,q_{ij})$ and $\theta_s$
to find a twisted generating function $(b',f^s_i,q'_{ij})$ which is linear
at infinity for $\varphi\times (z\pm s)$. As soon as $s\geq s_1$, $0$ is a regular value
of $f_i^s(x,\cdot)$ for all $x\in M_i$ and the topology of the sublevel set $\{f_i^s\leq 0\}$ is stable.
So it is enough to prove the result for $s=3s_1$ (for $s\geq 3s_1$ we can apply Lemma \ref{lem:homologydifference}).

According to Proposition~\ref{prop:tuberecognition}, up to stabilizing all $f^{3s_1}_i$ (by left acting with a quadratic form),
it is enough to show that $H_*(f^{3s_1}_i(x)\leq 0,f^{3s_1}_i(x)\leq-\infty)\simeq \Z[d_i]$
for any given $x\in M_i$ and some $d_i$.  First consider the difference function
$(\delta f^{3s_1}_i, \delta q'_{ij})$ and convert it to an actual function
$F\colon M\times \R^N\to \R$ using Lemma \ref{lem:difference}.

The only global critical points of $F$ with critical values less than or equal to $-4s_1$
consist of a copy of $L$ in $\{F=-6s_1\}$ in Morse-Bott
situation. After possibly passing to a double cover (of $M$ and hence of $L$,
recall $L\to M$ is a homotopy equivalence), we may assume the negative
eigenbundle of this critical submanifold is orientable and conclude
$H_{*+d}(F\leq -4s_1,F\leq -\infty)\simeq H_*(L)\simeq H_*(M)$ for some $d\in \Z$.

Now for any prime $p$, $C_{*+d}(F(x,-)\leq -4s_1,F(x,-)\leq -\infty; \Z/p\Z)$
defines a (derived) local system $G$ on $M$ (that is, an object of $D^b((\Z/p\Z)_M)$
with locally constant cohomology sheaves). More precisely,
we can describe $G$ as the dual of the pushforward of the constant sheaf on some sublevel set of $F$, namely
\[G=R\mathcal{H}om(R\pi_*(( \Z/p\Z)_{\{-A<F(x,-)\leq -4s_1\}}),(\Z/p\Z)_M[\dim M+d])\]
for large enough $A>0$. After passing to a finite cover we can assume that the cohomology sheaves
$H^i G$ are constant. By composition of derived functors, $R\Gamma(G)$ computes the relative
homology $H_{*+d}(F\leq -4s_1,F\leq -\infty)\simeq H_*(M)$, Lemma \ref{lem:rankone} then implies $G\simeq (\Z/p\Z)_M$.
In particular, for any point $x\in M$ and any prime $p$,  $H_{*+d}(F(x,-)\leq -4s_1,F(x,-)\leq -\infty;\Z/p\Z)\simeq \Z/p\Z$.
Finally, Lemma \ref{lem:homologydifference} (applied with $c=2s_1$) and Lemma~\ref{lem:tensor}
implies that for all $i$ and $x\in M_i$, $H_*(f^{3s_1}_i(x)\leq 0,f^{3s_1}_i(x)\leq-\infty)\simeq \Z[d_i]$
for some $d_i\in \Z$, which concludes the proof.
\end{proof}

\section{The stable Gauss map of nearby Lagrangians}\label{sec:tubespaces}

As we shall recall below, in \cite{waldhausen_algebraic_1982} Waldhausen defined a stable tube space $\WW_\infty$ (in there denoted $\TT_\infty$) and the rigid tube map $BO \to \WW_\infty$,
which was later proved by Bökstedt to be a rational homotopy equivalence.
In this section, we relate generating functions of tube type to Waldhausen's rigid tube map
and prove Theorem~\ref{thmintro:gaussmap} and Theorem~\ref{thmintro:homotopysphere}
using Bökstedt's rational equivalence result.

\subsection{Tube spaces}

Let $\TT_n$ be the set of pairs $(f,b)$ where $f:\R\times \R^n\to \R$ is a function of tube type and $b\geq 1$ is a bound for $f$.
Recall the function $f$ has the form
\[f(w,v)=w+g(v)+\epsilon(w,v),\]
with $\epsilon$ compactly supported, so
\begin{equation}
    \TT_n \subset C^\infty(\R^n,\R) \times C^\infty_c(\R\times \R^n,\R)\times (0,+\infty).
\end{equation}
We endow $C^\infty(\R^n,\R)$ with the weak $C^\infty$-topology, $C^\infty_c(\R\times \R^n,\R)$ with
the strong $C^\infty$-topology. Equipping $\TT_n$ with the restriction of the product topology we find that, if $X$ is a compact space and $X\to \TT_n$
is a continuous map, then the corresponding family $(\epsilon_x)_{x\in X}$ has a common compact support $K$ in $\R \times \R^n$
while the family $(g_x)_{x\in X}$ is allowed to vary outside of any compact set. But since 
\[\frac{\del f_x}{\del w}=\frac{\del (w+g_x+\epsilon_x)}{\del w}=1+\frac{\del \epsilon_x}{\del w}\]
is equal to $1$ outside of $K$, the critical points of $f_x$ do not escape to infinity.
So this topology on $\TT_n$ is appropriate to describe tube spaces. We also equip the space $\TT=\sqcup_n \TT_n$
with the disjoint union topology.

The operation $\oplus_b$ from Definition \ref{dfn:plusb} restricts to a continuous right action on $\TT$: if $t=(b,f)$, then $ t \cdot q=(b,f\oplus_b q)$. The right action uses the decomposition $\R^{n+m} = \R^{n} \times \R^{m}$. Using instead the decomposition $\R^{m+n} = \R^m \times \R^n$, we obtain a left action. Inspecting Equation \eqref{eq:modified_action} implies:
\begin{lem}
  The left and right actions of $\QQ$ on $\TT$ commute. \qed
\end{lem}

Using the left action by $h$, we define the limit space:
\[\TT_\infty=\colim(\TT \overset{h\cdot}{\longrightarrow} \TT \overset{h\cdot}{\longrightarrow} \cdots),\]
namely $\TT_\infty=\TT\times \N/\sim$ where $(t,i)\sim (h\oplus_b t, i+1)$.

We fix the basepoint for $\TT_\infty$ given by the pair $((D,4),0)$ at the first place of the colimit, where $D$ is the function from Figure \ref{fig:function-D-graph} (for which $b=4$ is a bound).
Note that $\TT_\infty$ inherits a right action of $\QQ$ since the two actions commute, and there is a $\QQ$-equivariant map $\TT_\infty \to \Z$ which takes $(t,i)$ with $t\in \TT_n$ to $n-2i$, where $\QQ$ acts on $\Z$ via the homomorphism $\dim \colon \QQ \to \N$. The base point lies in $\TT_0 \times \{0\}$, hence projects to $0$ under this map.

Finally, for any $d\in \Z$, we define the shift homeomorphism $\TT_\infty \to \TT_\infty$ to be given by $(t,i)[2d]=(h^d \oplus_b  t, i)$
for $d\geq 0$ and $(t,i)[2d]=(t,i-d)$ for $d\leq 0$. This corresponds to a translation of $2d$ in $\Z$
under the above map $\TT_\infty\to \Z$.

\begin{lem}\label{lem:actionequivalence}
For all $q\in \QQ$ the right actions $\TT_\infty \overset{\cdot q}{\longrightarrow} \TT_\infty$ and $\QQ_\infty \overset{\cdot q}{\longrightarrow} \QQ_\infty$
are homotopy equivalences.
\end{lem}
\begin{proof}
We only prove the case of $\TT_\infty$ as the other case is similar and easier. We claim that a homotopy inverse is given by the map $(t,i)\mapsto (t\oplus_b (-q), i)[2\dim q]$. First, from Lemma~\ref{lem:deltaq}
we know that from $q\oplus (-q)$ is homotopic to $h^{\dim(q)}$ (after a permutation of variables). Since
$h$ is preserved by an odd permutation ($(x,y)\mapsto(y,x)$), we obtain that $q\oplus (-q)$ is homotopic to $h^{\dim q}$.
Moreover the left and right actions by $h^{\dim(q)}$ are homotopic on $\TT$ since the permutation rearranging the factors
is even. Finally, the map induced on homotopy groups of the limit space $\TT_\infty$ can be represented on $\TT$,
and we obtain the weak (and hence strong since all spaces
are $CW$-complexes) homotopy equivalence result.
\end{proof}

Let $\TT_\infty'$ denote the fiber over $0$ of the projection map $\TT_\infty \to \Z$. Due to Lemma~\ref{lem:actionequivalence}
and \cite[Lemma D.1]{hatcher_short_2014}, we obtain a fibration sequence:

\[\TT'_\infty \to |B(\TT_\infty,\QQ)|\to |B(\Z,\QQ)|.\]

Simarly we denote by $\QQ_\infty'$ the fiber of $\QQ_\infty\to \Z$ over $0$. By right acting
on the base point $(D,4)$ we have a map $\QQ\to \TT$ and a map induced on the limits $\QQ_\infty\to \TT_\infty$.
This yields a map of fibration sequences:
\begin{equation}\label{finaldiagram}
\begin{tikzcd}
& & \star\ar[d,"\wr"] & \\
\Omega {|B(\Z,\QQ)|}\ar[r,"\sim"] \ar[d,equal] &\QQ_\infty' \ar[r]\ar[d] & {|B(\QQ_\infty,\QQ)|} \ar[r] \ar[d]& {|B(\Z,\QQ)|}\ar[d,equal]\\
\Omega {|B(\Z,\QQ)|}\ar[r] &\TT_\infty' \ar[r] & {|B(\TT_\infty,\QQ)|} \ar[r] & {|B(\Z,\QQ)|}
\end{tikzcd}
\end{equation}

The twisting datum of a twisted generating function is recovered using the map $\TT_\infty\to \Z$:

\begin{prop}\label{prop:liftgauss}
Let $(b,f_i,q_{ij})$ be a twisted generating function of tube type for a Legendrian
immersion $L\to J^1 M$ twisted by a map $h\colon M \to |B(\Z,\QQ)|$.
Then $h$ lifts through $|B(\TT_\infty,\QQ)|$.
\end{prop}
\begin{proof}
A twisted generating function of tube type gives naturally a simplicial map
$MV(M_\bullet)\to B(\TT,\QQ)$ which we compose with $B(\TT,\QQ)\to B(\TT_\infty,\QQ)$
by taking the first place in the colimit. The composition with $\TT_\infty\to \Z$
gives a map $M\simeq|MV(M_\bullet)|\to |B(\Z,\QQ)|$ which is homotopic to $h$ by definition.
\end{proof}

In order to relate our constructions to those of Waldhausen and Bökstedt, we now introduce Waldhausen's stable tube space $\WW_\infty$, and its relationship with the space $\TT_\infty$ of tube-like functions. Waldhausen's definition of $\WW_\infty$ is essentially the following. First define $\WW_{n,m}$
as the space of hypersurfaces $F \subset \R\times \R^{n}$ which coincide with $\{w=0\}$ at infinity
and are isotopic to the attachment of a trivial single handle of index $m$ on top of $\{w=0\}$.
Then Waldhausen defines stabilization maps
\begin{align*}
  \underline{\sigma} : \WW_{n,m} \to \WW_{n+1,m} \quad \text{and} \quad  \overline{\sigma}: \WW_{n,m}\to \WW_{n+1,m+1}
\end{align*}
and the limit space
\[\WW_\infty=\varinjlim_{n,m}\WW_{n,m}\]
with respect to both stabilization maps.
To define $\underline{\sigma}$ Waldhausen considers the homotopy equivalent sub-space $\underline{\WW_n} \subset \WW_n$ on which $F$ is completely contained in the upper half space $[0,\infty)\times \R^n$. He then (essentially) defines
\begin{equation}\label{eq:sigmaunderline}
  \underline{\sigma}(F) = \partial \big( (\{\leq F\} \times [-1,1]) \cup ((-\infty,0] \times \R^{n+1}) \big)
\end{equation}
where $\{\leq F\}$ denotes the side of $F$ which contains points with negative values in the first coordinates. This is not smooth but Waldhausen argues that one can pick a contractible choice of transverse vector field to smoothen it (he describes how to smoothen in the appendix to \cite{waldhausen_algebraic_1982}). The choice of the interval $[-1,1]$ is arbitrary and can be changed freely. The definition of $\overline \sigma$ is analogous.
Finally, Waldhausen defines the rigid tube map
\begin{equation}
    BO \to \WW_\infty
\end{equation}
by mapping a vector space $V \subset \R^n$ in $\R \times \R^n$ to the hypersurface obtained by attaching a trivial handle along the unit sphere in $V$.

\begin{prop}\label{prop:rigidtube}
There are equivalences $\QQ_\infty'\simeq \Z\times BO$ and $\TT_\infty'\simeq \Z\times \WW_\infty$
under which the map $\QQ_\infty'\to \TT_\infty'$ is given by the product of the rigid tube map with the identity on the $\Z$-factor.
\end{prop}
\begin{proof}

Let $\underline{\TT_n} \subset \TT_n$ denote the homotopy equivalent subspace where the level-set $\{f=0\}$ is in $[0,\infty) \times \R^n$, and where the gradient of $f$ is transverse to $\{0\}\times \R^n$ (necessarily pointing to the positive side). We set $\WW_n=\sqcup_m \WW_{n,m}$ and define maps $w_n : \underline{\TT_n} \to \underline{\WW_n}$ by essentially taking the level-set $\{f=0\}$.  However this is not quite well-defined since such a hypersurface is of the form $\{w=g(v)\}$ at infinity (rather than $\{ w = 0\}$). But we can cut-off $g$ in the equation $w=g(v)$ outside a compact set without changing that the hypersurface is transverse to the gradient of $f$, and the choice of such cut-off is contractible. The map $w_n$ is then a homotopy equivalence, since the space of functions with a given regular level set is convex and so is the condition on the gradient near $\{0\}\times \R^n$.

To compare Waldhausen's stabilization with ours, note that the hypersurface $\underline\sigma\circ w_n(f,b)$ is transverse to the gradient of $f\oplus q$ with $q(v)=v^2$ independent on the choice of the interval. So we may replace $[-1,1]$ in \eqref{eq:sigmaunderline} by a small enough interval depending on $b$ to also have the gradient of $f\oplus_b q$ transverse to it. By then using the gradient flow of $f \oplus_b q$ we can produce a contractible choice of homotopy from $\underline{\sigma} \circ w_n(f,b)$ to $w_{n+1} (f\oplus_b q,b)$. Similarly, $\overline{\sigma}\circ w_n$ is homotopic to $w_{n+1} \circ( \cdot \oplus_b (-q))$ and it follows that $h\oplus_b \cdot$ is homotopic to $\overline{\sigma} \circ \underline{\sigma}$.

The $BO$ factor in $\QQ_\infty'$ corresponds to the negative eigenspace $V$ of a non-degenerate quadratic form $q$ on $\R^n$, and the rigid tube
associated to such a vector subspace of $\R^n$ can be chosen such that the gradient of $D \oplus_b q$ is transverse to it.
Hence Waldhausen's rigid tubes can be canonically deformed to our rigid tubes of the form $\{D\oplus_b q =0\}$.

The remaining $\Z$-factor in $\QQ_\infty'$ corresponds to the signature (and hence to the first Maslov class).
This factor is also present in $\TT_\infty'$ since the index and coindex of the handle are
determined homologically, namely if $D\oplus_b q$ is homotopic in $\TT$ to $D\oplus_b q'$
then $q$ and $q'$ have the same signature.
\end{proof}

\subsection{A corollary of Bökstedt's theorem}

In \cite{bokstedt_rational_1984} Bökstedt proved the following theorem.
\begin{thm}[Bökstedt]
Waldhausen's rigid tube map $BO\to \WW_\infty$ is a rational homotopy equivalence.
\end{thm}

In \cite{kragh_generating_2018} it was argued why that actually implies that it is injective on homotopy groups.
For the convenience of the reader we recall the proof as it is rather short.

\begin{cor} \label{cor:tube-injective}
  Waldhausen's rigid tube map $BO\to \WW_\infty$ is injective on homotopy groups.
\end{cor}
\begin{proof}
We distinguish three cases:
\begin{itemize}
\item For $i=3,5,6, 7 \pmod 8$, $\pi_i BO=0$ and the statement is vacuous.
\item For $i=0,4 \pmod 8$, $\pi_i BO = \Z$ and the injectivity follows from the rational equivalence result of Bökstedt.
\item For $i=1,2 \pmod 8$, $\pi_i BO =\Z/2$ and the map $\pi_i BO \to \pi_i BG$ is injective by a result of Adams (here $BG$ denotes the
classifying space of spherical fibrations and $BO\to BG$ is the $J$-homomorphism). The statement follows
since $BO \to BG$ factors as the rigid tube map composed with the forgetful map $\WW_\infty \to BG$ mapping a tube to its underlying spherical fibration.
\end{itemize}
\end{proof}

From Diagram~\eqref{finaldiagram} and Proposition~\ref{prop:rigidtube}, we deduce the following corollary
which is the key input for the proof of Theorems~\ref{thmintro:gaussmap} and~\ref{thmintro:homotopysphere}.

\begin{cor}\label{cor:vanishing}
The map $|B(\TT_\infty,\QQ)|\to |B(\Z,\QQ)|$ is zero on all homotopy groups.
\end{cor}
\begin{proof}
$|B(\Z,\QQ)|$ is connected so the result holds on $\pi_0$.
For $k\geq 1$, the map $\pi_k |B(\TT_\infty,\QQ)| \to \pi_k |B(\Z,\QQ)|$ is
the same as the map $\pi_{k-1}\Omega|B(\TT_\infty,\QQ)| \to \pi_{k-1} \Omega|B(\Z,\QQ)|$,
which is zero in view of Diagram~\eqref{finaldiagram} and Corollary~\ref{cor:tube-injective} which says
that the map $\pi_{k-1}\QQ_\infty'\to \pi_{k-1}\TT_\infty'$ is injective.
\end{proof}

\subsection{Proof of Theorems~\ref{thmintro:gaussmap}, \ref{thmintro:gaussJ} and~\ref{thmintro:homotopysphere}} 

\begin{proof}[Proof of Theorem~\ref{thmintro:gaussmap}]
Pick a Legendrian lift $\varphi\times z \colon L\to J^1 M$ and $s>\max_L|z|$.
According to Theorem~\ref{thm:existencetgftube}, the $s$-double of $\varphi\times z$ admits a twisted
generating function of tube type twisted by $h:M \to B(\Z,\QQ)$ where $h\circ\pi$
is the Gauss map $g_\varphi$ of $L$. Then Proposition~\ref{prop:liftgauss} above says there is a lift of $h$
through $|B(\TT_\infty, \QQ)| \to |B(\Z,\QQ)|$. Now Corollary~\ref{cor:vanishing} implies that
$h$ is zero on homotopy groups and thus also $g_\varphi$.
\end{proof}

\begin{proof}[Proof of Theorem~\ref{thmintro:gaussJ}]
There is a map $\TT_\infty \to BG$ which classifies the spherical fibration associated with the sublevel set $\{f\leq 0\}$
of a function of tube type $f$. This can be made compatible with the $\QQ$-action, where on the right-hand side,
$\QQ$ acts via the $J$-homomorphism which takes $q\in \QQ$ to the sphere $S^-(q)$ (one point-compactification of the negative eigenspace).
Hence we obtain a map $|B(\TT_\infty,\QQ)|\to BG/BO=B(G/O)$ and the homotopy commutative diagram:
\begin{center}
\begin{tikzcd}
{|B(\TT_\infty,\QQ)|} \arrow[r]\arrow[d]& B(G/O) \arrow[d]\\
{|B(\Z,\QQ)|} \arrow[r,"\sim"] & B(\Z\times BO) \arrow[d] \\
 & B(\Z\times BG)
\end{tikzcd}
\end{center}
As in the proof of Theorem~\ref{thmintro:gaussmap}, a twisted generating function of tube type provides a lift $M\to |B(\TT_\infty,\QQ)|$ of the
Gauss map $L\to |B(\Z,\QQ)|$ (under the homotopy equivalence $\pi: L\to M$), and thus also a lift to $B(G/O)$ from which
we conclude that the composition to $B(\Z\times BG)$ is null-homotopic.
\end{proof}

\begin{proof}[Proof of Theorem~\ref{thmintro:homotopysphere}]
By Theorem~\ref{thmintro:gaussmap} the stable Gauss map of $L$ is nullhomotopic since $L$ is a homotopy sphere, hence Corollary~\ref{stable-trivial-linear} says its $s$-double for small $s$ has a generating function linear at infinity. Using Theorem~\ref{thm:chekanovEG} we may turn this into a generating function linear at infinity for the $s$-double for any $s>0$. By the same argument as in the proof of Theorem~\ref{thm:existencetgftube} this is tube like for large $s$, and the generating function obtained by restricting to the tube associated to all values below 0 generates only one copy of $L$.
\end{proof}

\appendix

\section{Monoids and bundles} \label{app:monoids-and-bundles}
This appendix studies principal bundles of topological monoids and their classifying spaces, as well as associated bundles in this context. While such notions have been studied before in the literature from different points of view \cite{Segal1978,MadsenMilgram1979}, we give a self-contained account, relying on a combination of \v{C}ech and simplicial methods.

\begin{rem} \label{Nice-Assumptions}
  In this paper all spaces will have the homotopy type of CW complexes, which means weak homotopy equivalences are homotopy equivalences. All simplicial spaces will also satisfy that face and degeneracy maps are cofibrations, which implies that the geometric realizations are of CW homotopy type but also that level wise homotopy equivalences induce homotopy equivalences (for both these statements see e.g. \cite{Segal1973}).
\end{rem}

\subsection{Directed open covers and MV-maps}
\label{sec:directed-open-covers}
\begin{dfn}\label{dfn:directedcover}
  A \emph{directed open cover} of a space $X$ consists of the data of a partially ordered set
  $(I,\leq)$ and an open covering $(U_i)_{i\in I}$ indexed by $I$ such that for all
  $x\in X$ the set $I_x=\{i\in I; x\in U_i\}$ is finite and totally ordered and each $U_i$ is non-empty.
\end{dfn}

There is an obvious way to associate to a directed open cover of the target of a continuous map, a directed open cover of its source by pullback (one discards indices for resulting empty sets); in particular, one may restrict such covers to subsets of $X$.

\begin{dfn}  \label{dfn:refinment_directed_cover}
A \emph{refinement} of a directed open cover $(U_i)_{i\in I}$ consists of a directed open cover $(V_j)_{j\in J}$
together with a map $\sigma\colon J\to I$ such that for all $j\in J$, $V_j\subset U_{\sigma(j)}$
and for all $x\in X$ the restriction $\sigma\colon J_x\to I_x$
is non-decreasing.
\end{dfn}

Given a directed open cover $(U_i)_{i\in I}$ we denote the intersection of the open sets associated to a totally ordered sequence $i_0\leq\dots\leq i_n$ by 
\begin{equation}
  U_{i_0\dots i_n} = U_{i_0} \cap \cdots \cap U_{i_n}.
\end{equation}
From this construction, we obtain a simplicial space $MV((U_i), i\in I)$ (\emph{the Mayer-Vietoris blow-up}, denoted $MV(U_\bullet)$ for short)
defined as follows: the space of $n$ simplices for $n\in \N$ is defined to be 
\begin{equation}
  MV(U_\bullet)_n=\bigsqcup_{i_0\leq\dots\leq i_n} U_{i_0\dots i_n}
\end{equation}
and a non-decreasing map $\alpha\colon [n]\to[m]$ induces the continuous map
\begin{equation}
  \alpha^*\colon MV(U_\bullet)_m\to MV(U_\bullet)_n
\end{equation}
defined by the inclusions
$U_{i_0\dots i_m}\to U_{i_{\alpha(0)}\dots i_{\alpha(n)}}$ for each
$i_0\leq\dots\leq i_m$. Note that there is an obvious map
$\pi\colon|MV(U_\bullet)|\to X$.  We recall that the realization of a simplicial
space $Z$ is $|Z| =\bigsqcup_{n\in \N} (Z_n \times \Delta_n)/\sim$, where
$\Delta_n$ is the standard $n$-simplex and $\sim$ is the equivalence relation
generated by $(z, \alpha_*t) \sim (\alpha^*z,t)$ for any $z \in Z_n$,
$t\in \Delta_m$ and $\alpha\colon [m]\to[n]$ non-decreasing.  We recall that a
simplex $z \in Z_n$ is degenerate if it belongs to $\im \alpha^*$ for some
$\alpha \colon [n] \to [m]$ with $m<n$.  Denoting by $Z_n^{nd}$ the subspace of
$Z_n$ of non-degenerate simplices, we have
$|Z| = \bigsqcup_{n\in \N} (Z_n^{nd} \times \Int \Delta_n)$.

\begin{prop} \label{MV-contractible-choice}
For any directed open cover $(U_i)_{i\in I}$ of $X$, the map 
\begin{equation}\pi \colon |MV(U_\bullet)|\to X\end{equation}
is a Serre fibration with contractible fiber, hence a homotopy equivalence.
\end{prop}
\begin{proof}
  For a given $x\in X$ the fiber $\pi^{-1}(x)$ is canonically homeomorphic to
  $\tilde x \times \Delta_p$, where $\tilde x \in MV(U_\bullet)_p$ is the
  non-degenerate simplex defined by $x \in U_{i_0\dots i_p}$ with $i_0<\dots<i_p$ and
  $\{i_0,\dots,i_p\} = \{i\in I; x\in U_i\}$.  We write the barycentric coordinates
  of $y\in \pi^{-1}(x)$ as $\rho_{i_j}(y)\in[0,1]$, $0\leq j \leq p$, with
  $\sum_j \rho_{i_j}(y)=1$.  Hence the fiber $\pi^{-1}(x)$ is a simplex (hence
  contractible), and (local) sections of $\pi$ correspond precisely to partitions of
  unity (i.e. functions $\rho_i\colon X\to [0,1]$, $\supp(\rho_i)\subset U_i$,
  $\sum_i \rho_i=1$).  Lifting a map $E\to X$ to $|MV(U_\bullet)|$ amounts to giving
  a partition of unity on $E$ with respect to the pullback open cover.  Checking that
  $\pi$ is a Serre fibration now boils down to the fact that a partition of unity on
  $D^k$ can be extended to $D^k\times [0,1]$, with respect to a given open cover of
  $D^k\times[0,1]$.
\end{proof}

\begin{dfn}
  A simplicial set $Z$ is \emph{directed} if it is equipped with a partial order
  $\leq$ on its set of vertices $Z_0$ such that for each simplex $\sigma \in Z_n$, we
  have $\alpha_0^* \sigma \leq \dots \leq \alpha_n^* \sigma$, where
  $\alpha_i\colon [0]\to[n]$ maps $0$ to $i$, and the map $Z_n \to Z_0^{n+1}$,
  $\sigma \mapsto (\alpha_0^* \sigma , \dots , \alpha_n^* \sigma)$ is injective.
\end{dfn}

When $Z_n \to Z_0^{n+1}$ is injective, a simplex $\sigma \in Z_n$ with
$\alpha_i^*(\sigma) = \alpha_{i+1}^*(\sigma)$ is degenerate.  In other words, if
$\sigma$ is non-degenerate, then $\alpha_0^* \sigma < \dots < \alpha_n^* \sigma$.

For a simplicial set $Z$, we write 
\begin{equation}
    q\colon \bigsqcup_{n\in \N} (Z_n \times \Delta_n) \to |Z|
\end{equation}
for the quotient map. We then associate to each simplex $\sigma \in Z_m$ the associated subset $|\sigma| \subset  |Z|$ which is the image of $\{\sigma\} \times \Delta_m $. We denote the interior of this geometric simplex by $\Int |\sigma| $, and the {\em star of $\sigma$} by
\begin{equation}
   \st_Z(\sigma) = \bigcup_{n\in \N} \bigcup_{\tau \in Z_n(\sigma)} \Int |\tau| 
\end{equation}
where $Z_n(\sigma)$ consists of all  $n$-simplices through which $\sigma$ factors.
Since it is a union of open sets, $\st_Z(\sigma)$ is an open subset of $|Z|$ (we write $\st(\sigma)$ if there is no ambiguity).  We have a
partition
\begin{equation}
  \label{eq:dec_star}
\st_Z(\sigma) = \bigsqcup_{n\in \N} \bigsqcup_{\tau \in Z^{nd}_n(\sigma)} \Int |\tau|
\end{equation}
where $Z^{nd}_n(\sigma)$ is the subset of $Z_n(\sigma)$ consisting of non-degenerate
simplices.

\begin{prop}\label{prop:starcover}
  Let $Z$ be a directed simplicial set. The stars of the vertices
  $(\st_Z(i))_{i\in Z_0}$ form a directed open cover of the geometric realization
  $|Z|$.  For $\sigma\in Z_n$ we have
  $\st_Z(\sigma) = \bigcap_{i=0}^n \st_Z(\alpha_i^*(\sigma))$, where
  $\alpha_i\colon [0]\to[n]$ maps $0$ to $i$.  Denoting
  $MV(\st_Z) = MV(\st_Z(i), i\in Z_0)$ for short, we have natural identifications
  $(MV(\st_Z))_n = \bigsqcup_{\sigma \in Z_n} \st(\sigma)$, for $n\in \N$, and
  \begin{equation}
  |MV(\st_Z)|
  =(\bigsqcup_{n\in \N}\bigsqcup_{\sigma \in Z_n} \st(\sigma) \times \Delta_n)/\sim,
\end{equation}
where $(x,\alpha_*t)\sim (i_\alpha(x),t)$ and $i_\alpha$ is the inclusion
$\st(\sigma)\to \st(\alpha^*\sigma)$.
\end{prop}
\begin{proof}
  For $x \in |Z|$, there exists a simplex $\sigma \in Z_n$ such that $x$ belongs to
  $\Int |\sigma|$. Then, for $i\in Z_0$, $x \in \st_Z(i)$ if and only if $i$ is a
  vertex of $\sigma$. By definition the vertices of $\sigma$ form a finite and
  totally ordered set.

  Any $\sigma \in Z_n$ can be writen $\sigma = \beta^*(\sigma')$, where $\sigma'$ is
  non-degenerate. Then $\sigma$ and $\sigma'$ have the same vertices and we have
  $\st(\sigma) = \st(\sigma')$. Hence to prove that
  $\st_Z(\sigma) = \bigcap_{i=0}^n \st_Z(\alpha_i^*(\sigma))$ we can assume that
  $\sigma$ is non-degenerate.  For $\tau \in Z_k$, if $|\tau|$ is contained in
  $\bigcap_{i=0}^n \st_Z(\alpha_i^*(\sigma))$, then all the vertices of $\sigma$ are
  vertices of $\tau$. By the injectivity of $Z_n \to Z_0^{n+1}$, this implies that
  $\sigma$ is a face of $\tau$, hence $|\tau| \subset \st(\sigma)$.  We deduce
  $\st_Z(\sigma) = \bigcap_{i=0}^n \st_Z(\alpha_i^*(\sigma))$. In turn this implies
  that $(MV(\st_Z))_n$ and $|MV(\st_Z)|$ are as described in the proposition.
\end{proof}

The natural homotopy equivalence $\pi\colon|MV(\st_Z)| \to |Z|$ (see Proposition
\ref{MV-contractible-choice}) has a canonical right inverse
\begin{equation}
  s\colon|Z| \to |MV(\st_Z)|
\end{equation}
defined as follows: for $\sigma \in Z_n$ and $t\in \Int \Delta_n$,
$s(\sigma,t)=(j_\sigma(t),t)$ where
$j_\sigma\colon \Int \Delta_n\to \st(\sigma)$ is the composition of the quotient
map $\Int \Delta_n \to \Int |\sigma|$ and the inclusion of
$\Int |\sigma| \to \st(\sigma)$. We can check that this defines a continuous map
(i.e. correctly glues at the boundary of simplices) and we have
$\pi\circ s=\id$; indeed for $x\in \st(\sigma)$ and $t\in \Delta_n$, we have
$\pi(x,t)=x \in |Z|$.

For any simplicial space $S$ we let $\diag \sing S$ denote the diagonal
simplicial set in the bi-simplicial set of degree-wise singular simplices in
$S$ (in degree $p$, this is the set of continuous maps from the simplex $\Delta^p$
to the space $S_p$). We have a natural map $|\diag \sing S| \to |S|$ which is a homotopy
equivalence (cf Remark~\ref{Nice-Assumptions}). For the next result, we write $\Hom$ for the set of morphisms of simplicial sets:

\begin{lem}\label{lem:MVB-simplicial}
  Let $Z$ be a directed simplicial set and $S$ a simplicial space which is degreewise Hausdorff.
There is a map $\Phi\colon \Hom(Z,\diag\sing S) \to \Hom(MV(\st_Z),S)$ such that
for each $f\colon Z \to \diag \sing S$ the canonical diagram
  \begin{center}
    \begin{tikzcd}
      {|MV(\st_Z)|} \ar[r, "|\Phi(f)|"] & {|S|} \\
      {|Z|} \ar[r, "|f|"] \ar[u, "\simeq", "s"'] & {|\diag \sing S|} \ar[u, "\simeq"] \\      
    \end{tikzcd}    
  \end{center}
  commutes.

Moreover, a simplicial map $g\colon MV(\st_Z) \to S$ is of the form $\Phi(f)$
if and only if for each $\sigma\in Z_n$ the map $g_n\colon \st(\sigma) \to S_n$
extends continuously to the closure $\overline{\st(\sigma)}$.

Finally, there is a map $\Psi\colon \im \Phi \to \Hom(Z,\diag\sing S)$ such that for each $g\in \im \Phi$, the diagram
  \begin{center}
    \begin{tikzcd}
      {|MV(\st_Z)|} \ar[r, "|g|"] & {|S|} \\
      {|Z|} \ar[r, "|\Psi(g)|"] \ar[u, "\simeq", "s"'] & {|\diag \sing S|} \ar[u, "\simeq"] \\      
    \end{tikzcd}    
  \end{center}
  commutes and moreover $\Phi\circ \Psi=\id$ on $\im \Phi$.
\end{lem}
\begin{proof}
  For a given $f\colon Z \to \diag \sing S$ we define $\Phi(f)$. We first define
  $\Phi^\sigma(f) \colon \st(\sigma)\to S_n$ for $\sigma \in Z_n$ using the
  decomposition~\eqref{eq:dec_star}.  Let $\tau \in Z_m$ be a non-degenerate simplex
  such that $\alpha^*\tau = \sigma$ for some $\alpha \colon [n] \to [m]$. Since
  $\tau$ is non-degenerate and $Z$ directed, $\alpha$ is unique.  We have
  $\Int |\tau| = \Int \Delta_m$ and we define
  $\Phi^\sigma_\tau(f) = \alpha^*\circ f(\tau) \colon \Int|\tau| \to S_n$, with
  $f(\tau)\colon \Int \Delta_m\to S_m$.  Using the fact that $f(\tau)$ is
  actually defined on $\Delta_m$, we can check that the $\Phi^\sigma_\tau(f)$'s give
  a continuous map $\Phi^\sigma(f) \colon \st(\sigma)\to S_n$.  Using
  $(MV(\st_Z))_n = \bigsqcup_{\sigma \in Z_n} \st(\sigma)$ (see
  Proposition~\ref{prop:starcover}) we obtain
  $\Phi(f)_n \colon (MV(\st_Z))_n \to S_n$.  We can check that this gives a
  simplicial map $\Phi(f) \colon MV(\st_Z)\to S$ with all requirements.  (The second
  statement follows from the already used fact that $f(\tau)$ is defined on
  $\Delta_m$.)

  Now for a given $g$ we define $\Psi(g)$.  For $\sigma \in Z_n$ we have a map
  $g_\sigma\colon \overline{\st(\sigma)}\to S_n$ and we define
  $\Psi(g)(\sigma)=g_\sigma \circ i_\sigma$ where $i_\sigma$ is the natural map
  $\Delta_n = |\sigma|\to \overline{\st \sigma}$.
\end{proof}

\begin{dfn}\label{dfn:equivMVmaps}
  An MV-map over $X$ is a map $g\colon MV(U_\bullet)\to S$ of simplicial spaces, for some directed
  open cover $(U_i)_{i\in I}$ of $X$.  We say that two maps
  $g_0\colon MV(U_\bullet)\to S$ and $g_1\colon MV(V_\bullet)\to S$ with respect
  to directed open covers $(U_i)_{i\in I}$ and $(V_j)_{j\in J}$ of $X$ are
  \emph{equivalent} if there is a directed open cover $(W_k)_{k\in K}$ of
  $X\times [0,1]$ a map $H\colon MV(W_\bullet)\to S$ such that $H_{\mid X\times \{i\}}=g_i$ (in particular, we assume that $(W_k)_{k\in K}$   restricts to $(U_i)_{i\in I}$ and $(V_j)_{j\in J}$ at the endpoints of the interval).
\end{dfn}

\begin{rem}
  If $f_0,f_1\colon X\to Y$ are homotopic maps and $MV(Y_\bullet)\to S$ is a
  simplicial map then the pullbacks $MV(f_0^{-1}(Y_\bullet))\to S$ and
  $MV(f_1^{-1}(Y_\bullet))\to S$ are equivalent.
\end{rem}

\begin{lem}
  Refining an MV-map gives an equivalent MV-map.
\end{lem}

\begin{proof}
  For this we construct an MV-map over $X\times[0,1]$ associated with a refinement. Let $g$
  be an MV-map indexed by a directed open cover $(U_i)_{i\in I}$ of $X$ and let
  $((V_j)_{j\in J},\sigma\colon J\to I)$ be a refinement of $(U_i)_{i\in I}$.
  Let $K=I\sqcup J$ with the order $k<l$ if ($k,l\in I$ and $k<l$) or
  ($k,l \in J$ and $k<l$) or ($k\in I$, $l\in J$, $k< \sigma(l)$) or ($k\in J$,
  $l\in I$, $\sigma(k)\leq l$). Then consider the open cover indexed by $K$
  given by $W_k=U_k\times [0,\frac{2}{3})$ for $k \in I$ and
  $W_k=V_k\times (\frac{1}{3},1]$ for $k\in J$.  We can check that it is a
  directed open cover and that the map $\sigma'\colon I\cup J\to I$, given by
  $\sigma'(i)=i$ for $i\in I$ and $\sigma'(j)=\sigma(j)$ for $j\in J$, defines a
  refinement of the directed open cover $(U_i\times[0,1])_{i\in I}$. Let $g'$ be
  the pullback of $g$ along the projection $X\times [0,1]\to X$ and
  $g''=(\sigma')^*g$ the pullback along the refinement $\sigma'$.  Then $g''$ is
  an MV-map over $X\times [0,1]$. Moreover $g$ and $\sigma^* g$ are refinements
  of the restrictions of $g''$ to $X\times \{0\}$ and $X\times \{1\}$
  respectively, along the inclusion maps $\sigma_0\colon I\to I\cup J$ and
  $\sigma_1\colon J\to I\cup J$.
\end{proof}

The following classification result is now a consequence of simplicial approximation.

\begin{cor} \label{cor:simp-approx}
  Let $X$ be a space, $S$ a simplicial space and $f\colon X\to |S|$ a continuous map.
There exists a directed open cover $(U_i)_{i\in I}$ of $X$ and a simplicial map
$g\colon MV(U_\bullet)\to S$ such that $f\circ \pi$ and $|g|$ are homotopic, where
$\pi\colon |MV(U_\bullet)|\to X$ is the natural projection. Moreover $g$ is unique up to equivalence.
\end{cor}

\begin{proof}
  Pick a directed simplicial approximation $\theta\colon |Z| \to X$ fine enough
  so that we can up to homotopy represent the composition
  
\[|Z| \xrightarrow{\theta} X \xrightarrow{f} |S| \simeq |\diag\sing S|\]
by a simplicial
  map $h\colon Z\to \diag\sing S$. Using Lemma~\ref{lem:MVB-simplicial} we
  obtain an MV-map $\Phi(h)\colon MV(\st_Z)\to S$ over $|Z|$. We choose a homotopy
  inverse $r\colon X\to |Z|$ of $\theta$ and define the MV-map over $X$ by the pull back of $\Phi(h)$ using $r$.
  
  Let $g_0$ and $g_1$ be two such maps. We pick a simplicial approximation
  $\theta\colon |Z|\to X$ so that the cover $(\overline{\st(i)})_{i\in Z_0}$
  refines both directed open covers pulled back by $\theta$.  Pick a homotopy
  inverse $r\colon X\to |Z|$. Since $\theta\circ r$ is homotopic to the
  identity, the pullback $(\theta\circ r)^*g_i$ is equivalent to $g_i$, and we
  are left to prove that $\theta^*g_0$ is equivalent to $\theta^*g_1$. According
  to Lemma \ref{lem:MVB-simplicial} there exists simplical maps
  $f_0,f_1\colon Z\to \diag \sing S$ such that $\theta^*g_0=\Phi(f_0)$,
  $\theta^*g_1=\Phi(f_1)$ and the geometric realizations $|f_0|$ and $|f_1|$ are
  homotopic.  By a relative version of simplicial approximation, we can find a
  subdivision $Y$ of $Z\times \Delta^1$ and a simplicial map $h\colon Y\to S$
  which extends $f_0$ and $f_1$. Then $\Phi(h)$ provides an equivalence between
  $\theta^*g_0$ and $\theta^*g_1$.
\end{proof}

\subsection{Principal bundles}

Let $Q$ be a unital associative monoid in spaces.

\begin{dfn}
  A $Q$-bundle $q$ indexed by a directed open cover $(U_i)_{i\in I}$ is a map $q_{ij}\colon U_{ij} \to Q$ for all $i<j$ such that
for all $i<j<k$, $q_{ij}q_{jk}=q_{ik}$ on $U_{ijk}$.
\end{dfn}

Any $Q$-bundle over $Y$ can be pulled back to $X$ along any continuous map $f\colon X\to Y$.
Also, one may restrict $Q$-bundles to any directed refinement.

\begin{dfn}  
  Two $Q$-bundles $((U^0_i)_{i\in I^0},q^0_{ij})$ and $((U^1_i)_{i\in I^1}, q^1_{ij})$ are equivalent
if they are the restriction of a $Q$-bundle on $X\times[0,1]$ to $X\times\{0\}$ and $X\times \{1\}$
respectively.
\end{dfn}

Recall also the bar construction of $Q$: define the simplicial space $BQ$ by \begin{equation}BQ_n=Q^n,\end{equation}
the faces for $0\leq i \leq n$, $d_i\colon Q^n\to Q^{n-1}$ given by $d_0(q_1,\dots, q_n)=(q_2,\dots,q_n)$, $d_n(q_1,\dots,q_n)=(q_1,\dots,q_{n-1})$
and $d_i(q_1,\dots,q_n)=(q_1,\dots,q_{i-1},q_iq_{i+1},\dots,q_n)$ if $0<i<n$,
degeneracies for $0\leq i \leq n$, $s_i\colon Q^n\to Q^{n+1}$ given by $s_i(q_1,\dots,q_n)=(q_1,\dots,q_i,1,q_{i+1},\dots,q_n)$.
More generally for a non-decreasing map $f\colon [n]\to [m]$, $f^*\colon Q^m\to Q^n$ is given by $f^*(q_1,\dots,q_m)=(p_1,\dots,p_n)$
with $p_i=\prod_{f(i-1)<k\leq f(i)} q_k$.

Observe that a simplicial map $E_.\to BQ$ is uniquely determined by the map
$h\colon E_1\to Q$. Namely the components of the map
$(q_1,\dots,q_n)\colon E_n\to Q^n$ are determined as follows: for $1\leq i\leq n$,
$q_i=h\circ f^*\circ(q_1,\dots,q_n)$ where $f\colon [1]\to[n]$ is defined by
$f(0)=i-1$ and $f(1)=i$. In particular a $Q$-bundle over $X$ with respect to a
directed open cover $(X_i)_{i\in I}$ is precisely the same data as a simplicial map
$g\colon MV(X_\bullet) \to BQ$.

The space $|BQ|$ is a classifying space for $Q$-bundles in the following sense.  A
$Q$-bundle with respect to a directed open cover $(X_i)_{i\in I}$ is a simplicial
map $MV(X_\bullet)\to BQ$. It induces a map on the geometric realizations
$f\colon |MV(X_\bullet)|\to |BQ|$.  Now the map $|MV(X_\bullet)|\to X$ is by
Lemma~\ref{MV-contractible-choice} a homotopy equivalence 
homotopy inverse $s\colon X\to |MV(X_\bullet)|$ defined up to a contractible choice, and we may consider
$f\circ s\colon X\to |BQ|$.  This associates to a $Q$-bundle an element of $[X,|BQ|]$
(homotopy classes of maps $X\to |BQ|)$, and we have:

\begin{prop}\label{classifying space}
  The above construction induces a bijection between equivalence classes of $Q$-bundles on $X$ and $[X,|BQ|]$.
\end{prop}

\begin{proof}
  This follows from Corollary \ref{cor:simp-approx} with $S=BQ$.
\end{proof}

\subsection{Associated bundles}\label{sec:app_associated_bundles}
Let $F$ be a right $Q$-space - i.e. a space with a right continuous action of the monoid $Q$.
We define a simplicial space $B(F,Q)$ as $B(F,Q)_n=F\times Q^n$ and for 
$f\colon [n]\to [m]$, $f^*\colon F\times Q^m\to F\times Q^n$ is given by $f^*(x,q_1,\dots,q_m)=(x',p_1,\dots,p_n)$
with 
\begin{equation} x'=x\left(\prod_{0<k\leq f(0)} q_k\right),\quad p_i=\prod_{f(i-1)<k\leq f(i)} q_k.\end{equation}

\begin{dfn} \label{dfn:twisted_map}
  A $Q$-twisted map to $F$ over $X$, $((X_i)_{i\in I}, f_i, q_{ij})$, consists of the
  data of a $Q$-bundle $((X_i)_{i\in I},q_{ij})$ and maps $f_i\colon X_i\to F$ such
  that $f_iq_{ij}=f_j$ on $X_{ij}$.
\end{dfn}

If $F$ is a point then this just recovers the notion of $Q$-bundle. Any map $F \to G$
of right $Q$-spaces associates to any $Q$-twisted map to $F$ a $Q$-twisted map to $G$. 
In particular the map $F \to \{*\}$ associates to a $Q$-twisted map to $F$ over $X$ a
normal $Q$-bundle (forgetting the maps $f_i$).
As for $Q$-bundles, a $Q$-twisted map to $F$ with respect to the
directed open covering $(X_i)_{i\in I}$ is the same as a simplicial map
$MV(X_\bullet)\to B(F,Q)$.  We define analogously the notion of equivalence between
$Q$-twisted maps to $F$.  A $Q$-twisted map to $F$ determines uniquely a
homotopy class of maps $X\to |B(F,Q)|$ and we have the following classification
result.

\begin{prop}\label{prop:QF-bundles}
  There is a natural bijection between equivalence classes of $Q$-twisted maps to
  $F$ over $X$ and $[X,|B(F,Q)|]$.
\end{prop}

\begin{proof}
  This follows from Corollary \ref{cor:simp-approx} with $S=B(F,Q)$.
\end{proof}

\begin{lem}\label{lem:BQQ}
For any topological monoid $Q$, the space $|B(Q,Q)|$ is contractible.
\end{lem}
\begin{proof}
For $n \in \N$, $B(Q,Q)_n=Q^{n+1}$.
It suffices to write down a simplicial homotopy between the simplicial
maps $B(Q,Q)\to B(Q,Q)$ given by the identity map and the constant map $c_e(q_0,\dots,q_n)=(e,\dots,e)$, where $e\in Q$ is the unit.
We define for $0\leq i\leq n$,
$h_i\colon Q^{n+1}\to Q^{n+2}$ by the formula
\begin{equation}h_i(q_0,\dots q_n)=(e,\dots,e,q_0\cdots q_{i},q_{i+1},\dots,q_n).\end{equation}
We can check that the following homotopy relations hold:
\begin{align}
&d_0h_0=\id,\quad d_{n+1}h_n=c_e,\\
&d_ih_j=h_{j-1}d_i \text{ if } i<j,\quad d_{j+1}h_{j+1}=d_{j+1}h_j,\quad d_ih_j=h_jd_{i-1} \text{ if } i>j+1,\\
&s_ih_j=h_{j+1}s_i \text{ if } i\leq j,\quad s_ih_j=h_js_{i-1} \text{ if } i>j.
\end{align}
\end{proof}

A map $f_i\colon X_i\to F$ is also a section of $F\times X \to X$ over $X_i$. Hence we
can reformulate the above notions by considering $Q$ acting on $F\times X$
preserving the projection $F\times X\to X$. We generalize this notion by
replacing $F\times X\to X$ by any $Q$-space $E$ over $X$ (here ``over'' means
with a map to $X$ for which $Q$ preserves fibers):
\begin{dfn} \label{dfn:Q-twisted-section}
  A  {\em $Q$-twisted section of $E$ over $X$} consists of the data of a $Q$-bundle $((X_i)_{i\in I},q_{ij})$ and sections $f_i\colon X_i\to E|X_i$ such that $f_iq_{ij}=f_j$ on $X_{ij}$.
\end{dfn}
The notion of $Q$-twisted map to
$F$ over $X$ now corresponds to a $Q$-twisted section of $F \times X$ over $X$.

\subsection{Total orderings}

\begin{lem}\label{lem:total-order}
  Any directed open cover on a smooth manifold $X$ admits a refinement where the order on $I$ is a total order.
\end{lem}

\begin{proof}
  Let $(U_i)_{i\in I}$ be a directed open cover of $M$.

  We pick a triangulation of $M$ so fine that the stars of the simplices refine $(U_i)_{i\in I}$. Denote
  by $J$ the set of all (non-degenerate) simplices in this triangulation and choose an arbitrary total order $\leq_0$ on $J$ and consider the
  lexicographic order $\leq$ given by dimension of the simplex and $\leq_0$, namely for two simplices $j, j'$,
  we have $j \leq j'$ if and only if $\dim j < \dim j'$ or $\dim j=\dim j'$ and $j \leq_0 j'$.
  Observe that in the barycentric subdivision, the vertices correspond bijectively to the original simplices.
  Now for $j\in J$, define $V_j$ to be the star of the $0$-simplex corresponding to $j$
  in the barycentric subdivision. We define $r\colon J\to I$ by the formula
  $r(j)=\max I_j$
  where \begin{equation}I_j=\{i\in I, \st(j)\subset U_i\}\end{equation}
  (here, $\st(j)$ refers to the original triangulation). Note that $I_j$ is finite and totally ordered and not empty since
  the cover $(\st(j))_{j\in J}$ refines $(U_i)_{i\in I}$, and hence the maximum makes sense.
  We have for all $j\in J$, $V_j\subset U_{r(j)}$ since $V_j\subset \st(j)$.
  Observe now that for all $x\in M$, the set $J_x=\{j\in J, x\in V_j\}$
  consists of simplices of distinct dimensions and that for $j,j'\in J_x$,
  if $\dim j<\dim j'$, $\st(j')\subset \st(j)$ and hence $r(j')\geq r(j)$.
  Hence $((V_j)_{j\in J},r)$ is a refinement of $(U_i)_{i\in I}$.
\end{proof}

\subsection{Shrinking}

\begin{dfn} \label{neg-shrink}
A \emph{shrinking} of an open cover $(X_i)_{i\in I}$ of a space $X$ is a sub-open cover $X'_i \subset X_i$, with the same indexing set. A set of shrinkings is \emph{negligible} if, given any closed subsets $C_i \subset X_i$, the given set includes a shrinking so that $C_i \subset X'_i$. 
\end{dfn}

Let $\cS$ be a set of closed subsets of a space $X$. We will denote all possible closed sets constructed by any finite number of combinations of taking boundaries, closure of complements, unions and intersections by $G(\cS)$.

\begin{lem} \label{cofibshrink}
  For any finite open cover $(M_i)_{i\in I}$ of a smooth manifold $M$ we can up to negligible shrinking assume that any inclusion $A \subset B$ with $A,B \in G(\{\overline{M_i}\}_{i\in I})$ is a cofibration.
\end{lem}

\begin{proof}
  By standard transversality theory we can assume that each open set in the shrinking has a smooth manifold as boundary and that any number of these boundaries intersects transversely.
\end{proof}

\subsection{Smoothing}

Let now $Q$ be a monoid in the category of smooth manifolds (i.e. the product map is smooth). Let also $E \to M$ be a smooth fiber bundle with $Q$ acting smoothly (preserving fibers as before). A \emph{smooth} $Q$-twisted section $(f_i,q_{ij},(M_i)_{i\in I})$ of $E$ is a $Q$-twisted section where the maps $f_i$ and $q_{ij}$ are smooth.

\begin{rem}\label{rem:smoothextension}
Let $X$ be a manifold and $A$ a closed subset of $X$. A function $f\colon A\to \R$
is said to be smooth if it can be extended to a smooth function $g$ on an open neighborhood $U$ of $A$.
Any smooth function on $A$ can be extended to a smooth function on $X$.
A germ of function on $A$ is an extension $(U,g)$ defined up to shrinking $U$.
Any germ of function on $A$ can be also extended to $X$.
\end{rem}

\begin{rem}\label{warning}
  We point out that if $f:A \to \R$ and $g:B \to \R$ is just smooth and they agree on $A\cap B$ then the glued function defined on $A\cup B$ need not be smooth. However, if there are germs of smooth functions defined on $A$ and $B$ which agree near $A\cap B$ then the extended germ is smooth.
\end{rem}

\begin{lem} \label{App:smoothing}
  If $Q$ and $E$ are smooth, then any $Q$-twisted section of $E$ is equivalent to a smooth one.
\end{lem}

\begin{proof}
  Let $(f_i,q_{ij},(M_i)_{i\in i})$ be a $Q$-twisted section. By Lemma~\ref{lem:total-order} we can assume that the order on $I$ is total. By Lemma~\ref{cofibshrink} we may shrink to assume that all sections are defined on the closures $\overline{M}_i$ and that all inclusions in the following are cofibrations.
  
  Assume for induction that we have picked a homotopy $f_i^t$ from $f_i$ to a (germ of a) smooth section for all $i<n$ and a homotopy $q_{ij}^t$ from $q_{ij}$ to a (germ of a) smooth map on $\overline{M}_{ij}:=\overline{M}_i\cap \overline{M}_j$ for all $i<j<n$ such that these satisfies $f_i^t q_{ij}^t=f_j^t$ and $q_{ij}^tq_{jk}^t = q_{ik}^t$ for all $t\in [0,1]$. The base case $n=1$ is vacuous and each induction step starts by picking homotopies $q^t_{in}$ downwards inductively in $i<n$. Indeed, when constructing $q^-_{in}$ we have already fixed what it should be on the cofibrant subset
  \begin{align*}
    \overline{M}_{in} \times \{0\} \cup \pare{\overline{M}_{in} \cap \pare{\cup_{i<j<n} \overline{M}_{jn}}} \times [0,1] \subset \overline{M}_{in} \times [0,1]
  \end{align*}
  as this is determined by the previously defined $q_{jn}^t$, $q_{in}^0$ and the equations $q_{in}^t=q_{ij}^tq_{jn}^t$. Note that, due to the order in which we are constructing these there are no requirements on what $q_{in}^t$ should be on the remaining part and we can extend freely. The last part of the induction step is similar. Indeed, $f_n^t$ is already determined on the cofibrant subset
  \begin{align*}
    \overline{M}_ n \times \{0\} \cup \pare{\cup_{i<n} \overline{M}_{in}} \times [0,1] \subset \overline{M}_n \times [0,1]
  \end{align*}
  by the equations $f_i^tq_{in}^t=f_n^t$, and free to be extended arbitrarily.
\end{proof}

\bibliographystyle{halpha}
\bibliography{large-bib}

\end{document}